\documentclass[10pt,twoside]{article}

\usepackage[utf8]{inputenc}
\usepackage[english]{babel}
\usepackage{amsthm}
\usepackage{blkarray}
\usepackage{tikz-cd}
\usepackage{amsfonts}
\usepackage{graphicx}
\usepackage[font=footnotesize]{caption}
\usepackage{floatrow}

\usepackage{amsmath}
\usepackage{array}
\usepackage{mathtools}
\usepackage{amssymb}
\usepackage{titling}
\usepackage{tikz}
\usepackage{float}
\usetikzlibrary{patterns}
\usepackage{authblk}
\usepackage{fancyhdr}

\usepackage{multirow}

\newcommand\blfootnote[1]{%
  \begingroup
  \renewcommand\thefootnote{}\footnote{#1}%
  \addtocounter{footnote}{-1}%
  \endgroup
}

\usepackage[backend=bibtex,style=numeric,natbib=true,doi=false,isbn=false,url=false,eprint=false,maxbibnames=9]{biblatex}
\usepackage[colorlinks=true,linkcolor=blue,citecolor=purple]{hyperref}
\usepackage[all]{xy}
\usepackage[ruled,vlined]{algorithm2e}

\SetKwInput{KwInput}{Input}                
\SetKwInput{KwOutput}{Output}              

 \newcommand\scalemath[2]{\scalebox{#1}{\mbox{\ensuremath{\displaystyle #2}}}}

 \newcommand{\subjclass}[2][1991]{%
  \let\@oldtitle\@title%
  \gdef\@title{\@oldtitle\footnotetext{#1 \emph{Mathematics subject classification.} #2}}%
}

\theoremstyle{definition}
\newtheorem{theorem}{Theorem}[section]
\newtheorem*{theorem*}{Theorem}
\newtheorem{definition*}{Definition}[]
\newtheorem{nono-definition}{Definition}[]
\newtheorem{proposition}{Proposition}[section]
\newtheorem{proposition*}{Proposition}
\newtheorem{definition}{Definition}[section]
\newtheorem{lemma}{Lemma}[section]
\newtheorem{example}{Example}[section]
\newtheorem{corollary}{Corollary}[section]
\newtheorem{notation}{Notation}[section]

\newtheorem{remark}{Remark}[section]

\newcommand{\lb}[1]{{#1}}
\newcommand{\lbb}[1]{{#1}}
\newcommand{\cc}[1]{{#1}}
\newcommand{\ccc}[1]{{#1}}

\numberwithin{equation}{section}

\newcommand\restr[2]{{
  \left.\kern-\nulldelimiterspace 
  #1 
  \vphantom{\big|} 
  \right|_{#2} 
  }}
 
\usepackage[margin=1.05in,footskip=0.15in,tmargin=1in,bmargin=0.85in]{geometry}

\makeatletter
\catcode`! 3
\def\Transpose #1{\romannumeral0\expandafter
                  \Mar@Transpose@a\romannumeral`^^@\Mar@DoOneRow #1\\!\\}

\def\Mar@DoOneRow #1\\{\Mar@DoOneRow@a {}#1&^^@&}%

\def\Mar@DoOneRow@a #1#2&{%
    \if^^@\detokenize{#2}\expandafter\@gobble\fi
    \Mar@DoOneRow@a {#1#2\\}%
}%

\def\Mar@Transpose@a #1#2\\{\ifx!#2\expandafter\Mar@FinishTranspose\fi
    \expandafter\Mar@Transpose@b\romannumeral`^^@\Mar@DoOneRow@a {}#2&^^@&#1}

\def\Mar@Transpose@b #1#2^^@\\{\Mar@Join {}#2^^@!#1}

\def\Mar@Join #1#2\\#3!#4\\%
   {\if^^@\detokenize{#3}\expandafter\Mar@EndJoin\fi
    \Mar@Join {#1#2&#4\\}#3!}%

\def\Mar@EndJoin\Mar@Join #1^^@!^^@\\{\Mar@Transpose@a {#1^^@\\}}

\def\Mar@FinishTranspose
    #1&^^@&#2\\^^@\\{ #2}

\catcode`! 12
\makeatother

\newcommand{\Addresses}{{
  \bigskip
  \footnotesize

  \textsc{Université Cote d’Azur, Inria, 2004 route des Lucioles, 06902 Sophia Antipolis, France.}\par\nopagebreak
  \textit{E-mail address}, \texttt{laurent.buse@inria.fr}

  \medskip

 \textsc{Athena RC, Artemidos 6 kai Epidavrou, 15125 Maroussi , Greece}\par\nopagebreak
  \textit{E-mail address}, (corresponding author): \texttt{ccheca@di.uoa.gr}
}}

\DeclareMathOperator{\Hom}{Hom}
\DeclareMathOperator{\Relint}{Relint}
\DeclareMathOperator{\HF}{HF}
\DeclareMathOperator{\HP}{HP}
\DeclareMathSymbol{\ast}{\mathbin}{symbols}{"03}

\DeclareMathOperator{\sylv}{sylv}
\DeclareMathOperator{\Sylv}{Sylv}

\DeclareMathOperator{\Res}{Res}
\DeclareMathOperator{\sat}{sat}

\DeclareMathOperator{\im}{Im}

\DeclareMathOperator{\Pic}{Pic}
\DeclareMathOperator{\Supp}{Supp}

\DeclareMathOperator{\Spec}{Spec}

\DeclareMathOperator{\Cl}{Cl}

\DeclareMathOperator{\Id}{Id}

\DeclareMathOperator{\Residue}{Residue}

\DeclareMathOperator{\Ker}{Ker}

\providecommand{\keywords}[1]
{
  \small	
  \textbf{\textit{Keywords: }} #1
}

\begin{document}

\fancyhead[CO]{\scriptsize Toric Sylvester forms }
\fancyhead[CE]{\scriptsize Laurent Busé and Carles Checa}
\fancyhead[LE]{\scriptsize\thepage}
\fancyhead[RO]{\scriptsize\thepage}
\fancyhead[LO]{}
\fancyhead[RE]{}
\fancyfoot[R]{}
\fancyfoot[C]{}
\fancyfoot[L]{}
\thispagestyle{empty}

\title{Toric Sylvester forms}

\author[1]{Laurent Busé}
\author[2]{Carles Checa}

\affil[1]{Université Côte d'Azur, Inria, France}
\affil[2]{National Kapodistrian University of Athens, Athena RC, Greece}

\maketitle

\begin{abstract}
In this paper, we investigate the structure of the saturation of ideals generated by sparse homogeneous polynomials over a projective toric variety $X$ with respect to the irrelevant ideal of $X$. 
As our main results, we establish a duality property and make it explicit by introducing toric Sylvester forms, under a certain positivity assumption on $X$. In particular, we prove that toric Sylvester forms yield bases of some graded components of $I^{\text{sat}}/I$, where $I$ denotes an ideal generated by $n+1$ generic forms, $n$ is the dimension of $X$ and $I^{\text{sat}}$ \ccc{is} the saturation of $I$ with respect to the irrelevant ideal of the Cox ring of $X$.
Then, to illustrate the relevance of toric Sylvester forms we provide three consequences in elimination theory over smooth toric varieties: (1) we introduce a new family of elimination matrices that can be used to solve sparse polynomial systems by means of linear algebra methods, including overdetermined polynomial systems; (2) by incorporating toric Sylvester forms to the classical Koszul complex associated to a polynomial system, we obtain new expressions of the sparse resultant as a determinant of a complex; (3) we \cc{explore the computation of the toric residue of the product of two forms.}
\end{abstract}


\blfootnote{\textup{2000} \textit{Mathematics Subject Classification}: \textup{05E14,14Q99  
}}

\section{Introduction}

The elimination of variables in a system of homogeneous polynomial equations is deeply connected to the saturation of ideals with respect to a certain geometrically irrelevant ideal. Thus, the search and study of universal generators of the saturation of an ideal generated by generic homogeneous polynomials is an important topic in elimination theory. In the classical literature of the previous century, such universal generators were called \textit{inertia forms} by Hurwitz, Mertens, Van der Waerden and many others, including Zariski; see the references in \cite{Jouanolou91,joanolouformesdinertie} and \cite{Zariski}. As examples, Jacobian determinants and resultants associated to a square homogeneous polynomial system are important inertia forms.

To be more specific, consider the ideal $I = (F_0,\dots,F_n )$ where $F_i$ is the generic homogeneous polynomial of degree $d_i$ in the graded polynomial ring $C = A[x_0,\dots,x_n]$, where $\deg(x_i)=1$ for all $i=0,\ldots,n$ and where $A$ stands for the universal ring of coefficients of the $F_i$'s. 
The saturation of the ideal $I$ with respect to the ideal $\mathfrak{m}=(x_0,\ldots,x_n)$, which we denote by $I^{\sat}=I:\mathfrak{m}^\infty$, is the ideal of inertia forms. In this context, the ideal $\mathfrak{m}$ is the (geometrically) irrelevant ideal of the projective space of dimension $n$ which is associated to $C$. As the elements in $I$ are trivially inertia forms, $I^{\sat}/I$ is the natural quotient to study. It turns out that the Jacobian determinant of the $F_i$'s is a generator, as an $A$-module, of the graded component of $I^{\sat}/I$ in degree $\delta = d_0 + \dots + d_n - (n+1)$ and their resultant is a generator of $I^{\sat}/I$ in degree 0. In order to unravel the structure of $I^{\sat}/I$ in degrees smaller than $\delta$, Jouanolou introduced and studied the formalism of Sylvester forms  \cite{joanolouformesdinertie}. His ideas were based on the fact that for each $\mu = (\mu_{0},\dots,\mu_n)\in \mathbb{N}^{n+1}$ such that $|\mu| := \sum_i \mu_i < \min_i d_i$, each polynomial $F_i$ can be decomposed as  
\begin{equation}\label{eq:decompFi}
F_i = \sum_{j=0}^n x_j^{\mu_j + 1}F_{i,j}
\end{equation}
and one can consider the determinant $\det(F_{i,j})_{0\leq i,j \leq n}$. This latter is called a \textit{Sylvester form}  of the $F_i$'s and denoted by $\Sylv_{\mu}$. Independently of the choice of decompositions \eqref{eq:decompFi}, the class of $\Sylv_{\mu}$ modulo $I$, which is denoted by $\sylv_{\mu}$, gives a nonzero element in $(I^{\sat}/I)_{\delta - |\mu|}$. Moreover, $(I^{\sat}/I)_{\delta - |\mu|}$ is a free $A$-module which can be generated by the Sylvester forms of degree $\delta - |\mu|$. This result is a consequence of a duality property between Sylvester forms and monomials; namely, for all $\nu<\min_i d_i$ we have an isomorphism of $A$-modules
\[ (I^{\sat}/I)_{\delta - \nu} \simeq \Hom_A(C_{\nu},A). \]
More explicitly, this isomorphism corresponds to the equalities
\[ x^{\mu'}\sylv_{\mu} = \begin{cases} \sylv_0 & \textrm{ if } \mu = \mu' \\ 0 & \textrm{ if } \mu \neq \mu'  \end{cases}\]
where $\sylv_0$ is a generator of $(I^{\sat}/I)_{\delta}$. We note that up to a nonzero multiplicative constant, $\sylv_0$ is equal to the class of the Jacobian determinant of the $F_i$'s; see \cite[\S 3.10]{joanolouformesdinertie}. 

The definition and main properties of Sylvester forms have been recently extended to the case of $n+1$ generic multi-homogeneous polynomials, i.e.~of polynomials defining hypersurfaces over a product of projective spaces of total dimension $n$; see \cite{buse2021multigraded}. In this paper, we develop the theory of Sylvester forms in the general setting of homogeneous polynomials in the coordinate ring of a projective toric variety $X_{\Sigma}$. In addition, to illustrate the importance of these forms in elimination theory, we also provide applications to the construction of elimination matrices for overdetermined polynomial systems and to the computation of toric resultants and toric residues. As far as we know, these applications are also new results in the context of multi-homogeneous polynomial systems. In what follows we give a brief overview of the main contributions in this paper. 

Let $\mathbf{k}$ \cc{be a field} and $X_{\Sigma}$ be a $n$-dimensional projective toric variety over $\mathbf{k}$ given by a complete fan $\Sigma$ in a lattice $N$. Let $R$ be the homogeneous coordinate ring of $X_{\Sigma}$ over $\mathbf{k}$, also known as the \textit{Cox ring} of $X_{\Sigma}$, which is graded using the combinatorics of $\Sigma$; see Section \ref{sec:toricgeom} or \cite{coxring} for more details. Assuming that there exists a smooth $n$-dimensional cone  $\sigma \in \Sigma(n)$, we write $x_1,\ldots,x_n$ the variables of $R$ associated to $\sigma$ and we denote by $z_1,\ldots,z_r$ the remaining ones. With these notations, 
a homogeneous polynomial in $X_{\Sigma}$ of degree $\alpha \in \Cl(X_{\Sigma})$, the class group of $X_{\Sigma}$, is an element in the graded component $R_\alpha$ of $R$ in degree $\alpha$; \cc{i.e.} a $\mathbf{k}$-linear combination of monomials $x^{\mu} = x_1^{\mu_1}\cdots x_n^{\mu_n}z_1^{\mu_{n+1}}\cdots z_r^{\mu_{n+r}}$ of degree $\alpha$ \ccc{(we notice that what we call homogeneous polynomials are also sometimes called multi-homogeneous polynomials)}. Now, the \textit{generic} homogeneous polynomial of degree $\alpha$ is the polynomial  
$\sum_{x^{\mu} \in R_\alpha}c_{i,\mu}x^\mu$
where the coefficients $c_{i,\mu}$ are seen as variables. Thus, being given $n+1$ degrees $\alpha_0,\ldots,\alpha_n$, the corresponding generic homogeneous polynomial system over $X_\Sigma$ is given by the $n+1$ homogeneous polynomials 
\begin{equation}\label{eq:Fiintro}
F_i = \sum_{x^{\mu} \in R_{\alpha_i}}c_{i,\mu}x^\mu \in C = A\otimes_{\mathbf{k}} R =  A[x_1,\dots,x_n,z_1,\dots,z_r], \ i=0,\ldots,n,
\end{equation}
where $A$ is the universal ring of coefficients over $\mathbf{k}$, i.e.~$A= \mathbf{k}[c_{i,\mu} \, : \, x^{\mu} \in R_{\alpha_i}, \ i = 0,\dots,n].$
We define the ideal $I=(F_0,\ldots,F_n)$ and the ideal $\mathfrak{b} = ( \tilde{x}^{\sigma} \,:\, \tilde{x}^{\sigma} = \prod_{\rho \notin \sigma}x_{\rho}, \, \sigma \in \Sigma(n))$, which is the \textit{irrelevant ideal} of $X_\Sigma$. The saturation of $I$ is the ideal of $C$ defined as $I^{\sat} = ( I : \mathfrak{b}^{\infty})$.  

\medskip

The first main result of this paper is the following duality property which is a generalization of \cite[Theorem A]{buse2021multigraded} to the case of a projective toric variety; see Theorem \ref{dualitytheorem}. We set 
$\delta = \alpha_0 + \dots + \alpha_n - K_X \in \Cl(X_{\Sigma})$, where $K_X$ is the anticanonical class of $X_{\Sigma}$.
\begin{theorem*}\label{dualitytheoremintro}
Let $X_{\Sigma}$ be a projective toric variety, let $\sigma \in \Sigma(n)$ be an $n$-dimensional smooth cone and let $\nu \in \Cl(X_{\Sigma})$. Then, with the above notation, there exists a non-empty region $\Gamma \subsetneq \Cl(X_{\Sigma})$ such that if $\delta - \nu \notin  \Gamma$, we have
\[ (I^{\sat}/I)_{\delta - \nu} \simeq \Hom_{A}((C/I)_{\nu},A).\]
\end{theorem*}

In the cases where $(C/I)_{\nu} = C_{\nu}$, the above duality implies that $(I^{\sat}/I)_{\delta - \nu}$ is a free $A$-module (see Corollary \ref{cor:freemodule}) and a natural question is to find explicit bases. To tackle this question, we introduce toric Sylvester forms. We first prove that for suitable $\nu$ and $x^{\mu} \in R_{\nu}$, each generic homogeneous polynomial $F_i$ can be decomposed into $n + 1$ generic homogeneous polynomials $(F_{ij}^{\mu})_{0\leq j\leq n}$, similarly to \eqref{eq:decompFi}. The existence of such decompositions requires a certain property on the smooth cone $\sigma \in \Sigma(n)$; when it holds we will say that $X_\Sigma$ is $\sigma$-\textit{positive} (see Definition \ref{def:definitionpositive} and Theorem \ref{theoremdecompositionrestated}). Then, from these decompositions we define toric Sylvester forms as the determinants $\Sylv_{\mu} := \det(F^{\mu}_{ij}) \in I^{\sat}_{\delta-\nu}$ and \cc{consider their classes in} $(I^{\sat}/I)_{\delta - \nu}$, denoted by  $\sylv_{\mu}$. Finally, we obtain the following explicit duality property which can be seen as the second main contribution of this paper; see Theorem \ref{sylvesterformduality} \cc{and Theorem \ref{sylvesterformsbasis}}.

\begin{theorem*} Let $X_{\Sigma}$ be a projective toric variety and let $\sigma \in \Sigma(n)$ be a smooth cone such that $X_\Sigma$ is $\sigma$-positive. Then, under suitable conditions on $\nu \in \Cl(X_{\Sigma})$, for any pair $x^\mu,x^{\mu'} \in R_{\nu}$ we have
\[ x^{\mu'}\sylv_{\mu} = \begin{cases} \sylv_0 & \textrm{ if } \mu = \mu'  \\ 0 & \textrm{ if } \mu < \mu',\end{cases}\]
where $\sylv_0$ is a generator of $(I^{\sat}/I)_{\delta}$ as an $A$-module and $\mu < \mu'$ denotes a total ordering of the corresponding monomials (see Definition \ref{definitionorder}). Therefore, for suitable $\nu \in \Cl(X_{\Sigma})$ the toric Sylvester forms $\{\sylv_{\mu}\}_{x^\mu \in R_{\nu}}$ yield an $A$-basis of $(I^{\sat}/I)_{\delta - \nu}$.
\end{theorem*}

In the rest of the paper, we provide three applications of toric Sylvester forms in elimination theory. The first application deals with \textit{elimination matrices}. An important question in elimination theory is the study of matrices $\mathbb{M}$ with entries in $A$ such that: 
\begin{itemize}
\item[i)] their rank drops when the coefficients $c_{i,\mu}$'s are specialized in $\mathbf{k}$ and the corresponding polynomial system has solutions in $X_{\Sigma}$,
\item[ii)] their corank coincides with the number of solutions (over $\mathbf{k}$), counting multiplicities, when the coefficients $c_{i,\mu}$'s are specialized in $\mathbf{k}$ and the corresponding polynomial system has finitely many solutions in $X_{\Sigma}$.
\end{itemize}
The first property is related to resultant theory (see e.g.~\cite{cattani1997residuesresultants,gkz1994})  whilst the second is used for solving $0$-dimensional polynomial systems (see e.g.~\cite{bender2021toric,emirismourrain}). In this paper, we introduce a new family of elimination matrices by adding to a classical Macaulay-block matrix in some degree $\alpha \in \Cl(X_{\Sigma})$, a block-matrix built from the toric Sylvester forms of degree $\alpha$ (see Definition \ref{def:H_alpha}). We call these matrices \textit{hybrid elimination matrices} and prove their main properties in Theorem \ref{prop:Helimmat}. Compared with the more classical Macaulay matrices, this new family yields more compact matrices that can still be used for solving 0-dimensional polynomial systems. In addition, we also prove that the construction of hybrid elimination matrices can be extended to polynomial systems defined by more than $n+1$ polynomials (see Theorem \ref{theoremoverdetermined}).

Our second application concerns the \textit{computation of sparse resultants}. 
A classical result in elimination theory is that, under certain assumptions, the sparse resultant can be computed as the determinant of certain graded components of the Koszul complex built from the considered polynomial system (see \cite[Chapter 3 and Chapter 8]{gkz1994}). Generalizing  a construction of Cattani, Dickenstein and Sturmfels in \cite[\S 2]{cattani1997residuesresultants} that uses the so-called toric Jacobian, we modify the usual Koszul complex by incorporating the Sylvester forms in its last differential and prove that the determinant of some suitable graded parts of this new complex is equal to the sparse resultant, up to a nonzero multiplicative constant in $\mathbf{k}$ (see Theorem \ref{resultantformula}). This result yields new formulas for computing the sparse resultant as a determinant of a complex.

As our third application, \cc{we explore the use of Sylvester forms} in \emph{\ccc{the} computation of toric residues}. The toric residue of the generic polynomial system \eqref{eq:Fiintro} was defined by Cox in \cite{coxtoricres}. It is a map that sends any polynomial in $(C/I)_\delta$ to the fraction field $K(A)$ of $A$. The computation of this residue map by means of determinants has been an active research topic with many contributions, including \cite{joanolouformesdinertie,dandreakhetanresidues,cattani1995residues}. In this paper, using toric Sylvester forms we construct matrices whose determinants are used to compute the residue of a product of two forms $PQ$, where $P \in C_{\nu}$, $Q \in C_{\delta - \nu}$ and  $\nu \in \Cl(X_{\Sigma})$. This formula can be seen as an extension of a similar formula proved by Jouanolou in the case $X_\Sigma=\mathbb{P}^n$ \cite[Proposition 3.10.27]{joanolouformesdinertie}.

The paper is organised as follows. In Section  \ref{sec:toricgeom}, we present all the tools of toric geometry that are needed in the rest of the paper. In particular, we prove the existence of decompositions of forms in a projective toric variety $X_{\Sigma}$ which is $\sigma$-positive for a smooth cone $\sigma \in \Sigma(n)$.  
In Section \ref{sec:duality}, we show that the claimed  duality property holds outside a region $\Gamma \subset \Cl(X_{\Sigma})$ which depends on the supports of the local cohomology modules of the corresponding Cox ring. In Section \ref{sec:sylvforms}, we define Sylvester forms and show that they give an $A$-basis of $(I^{\sat}/I)_{\delta - \nu}$ for certain degrees $\nu \in \Cl(X_{\Sigma})$. In Section \ref{sec::hybridmat}, we introduce hybrid elimination matrices when $X_{\Sigma}$ is assumed to be smooth and $\sigma$-positive for a smooth cone $\sigma \in \Sigma(n)$.
In Section \ref{sec:resultants}, we prove that the determinant of certain graded parts of a modified Koszul complex outside a region $\Gamma_{\Res} \subset \Cl(X_{\Sigma})$ is equal to the sparse resultant, up to a nonzero multiplicative constant in $\mathbf{k}$.  Finally, in Section \ref{sec:residues}, we explore the computation of the toric residue of a product of two forms.

\section{Preliminaries on toric geometry}\label{sec:toricgeom}

In this section, we fix our notation and briefly review some material we will use from toric geometry; we refer to the book by Cox, Little and Schenck \cite{coxlittleschneck} for more details. We also prove a decomposition property that we will use to introduce toric Sylvester forms later on.

\paragraph{Projective toric varieties.} 
Let $\mathbf{k}$ be a field and let $M$ be a lattice of rank $n$ ($\simeq \mathbb{Z}^n$). We denote by $N = \Hom_{\mathbb{Z}}(M, \mathbb{Z})$ the dual of $M$, by $\mathbb{T}_N=N\otimes \overline{\mathbf{k}}^{\times}$ the algebraic torus associated to $N$ over an algebraic closure of $\mathbf{k}$. We also set $M_{\mathbb{R}} = M \otimes \mathbb{R}$, which is a vector space over the real numbers. Let $\Delta$ be a very ample lattice polytope with vertices in $M$ and let $\mathcal{A} = \Delta \cap M$ be the lattice points in this polytope. The projective toric variety $X_{\Delta}$ can be defined as the Zariski closure in $\mathbb{P}^{s-1}$ of the image of the map 
\[ \Phi_{\mathcal{A}}: \mathbb{T}_N \xrightarrow{} \mathbb{P}^{s-1} \quad t:=(t_1,\ldots,t_n) \xrightarrow[]{} 
(t^{m_1}: \dots: t^{m_s}),\]
where $\{m_1,\dots,m_s\}$ are the lattice points in $\mathcal{A}$; see \cite[§2.3]{coxlittleschneck}. 
\begin{example}
If $\Delta$ is a product of simplices of the form $\Delta_{n_j} = \{ t \in \mathbb{R}^{n_j} \, :\,  t_i \geq 0, \, \sum_{i = 0}^{n_j} t_i \leq 1\}$ for $j = 1,\dots,s$, then $X_{\Delta} = \mathbb{P}^{n_1} \times \dots \times \mathbb{P}^{n_s}$.
\end{example}

 When $\Delta$ is $n$-dimensional, a more intrinsic definition of the same toric variety can be stated from the normal fan $\Sigma \subset N$ of $\Delta$, so that this variety is also denoted by $X_{\Sigma}$; see \cite[Proposition 3.1.6]{coxlittleschneck} for the equivalence between these definitions. The geometric properties of $X_{\Sigma}$ are deeply connected with the combinatorial properties of the fan $\Sigma$. For instance,  $X_{\Sigma}$ is a smooth variety if and only if $\Sigma$ is smooth, which means that \ccc{for each cone $\sigma \in \Sigma$, its minimal generators are part} of a basis of $N$. 

Moreover, if we are given any $n$-dimensional lattice polytope $\Delta$ (not necessarily very ample) the toric variety to consider is the one associated to $l\Delta$ for $l \gg 0$, which has the same normal fan and is very ample; see \cite[Definition 2.3.14]{coxlittleschneck}. Once we are given a lattice polytope $\Delta$, the variety $X_{\Delta}$ (or identically $X_{\Sigma}$) will be the ambient space throughout this paper. 

We denote by \ccc{$\Sigma(d)$} the set of \ccc{$d$}-dimensional cones of $\Sigma$, which are also called rays when \cc{$d=1$}. We assume that the generators of the rays $u_{\rho} \in N$ for $\rho \in \Sigma(1)$ are primitive and span the vector space $N_{\mathbb{R}}$; by \cite[Corollary 3.3.10]{coxlittleschneck}, this condition is equivalent to the toric variety $X_{\Sigma}$ having no torus factors. Moreover,  if  $\Delta$ is an $n$-dimensional bounded polytope, its normal fan $\Sigma$ is complete 
and the cones of this fan are strongly convex. Under these assumptions, $\Sigma(1)$ contains at least $n + 1$ rays. In addition, 
if we denote the class group of $X_{\Sigma}$ by $\Cl(X_{\Sigma})$, there is a short exact sequence
\begin{equation}
\label{eq:ses}
    0 \xrightarrow[]{} M \xrightarrow[]{\mathbf{F}} \mathbb{Z}^{\Sigma(1)} \xrightarrow[]{\pi} \Cl(X_{\Sigma}) \xrightarrow[]{} 0,
\end{equation} 
where $\mathbf{F}$ is an $(n + r) \times n$ matrix whose rows are the generators of the rays in $\Sigma(1)$ and $\pi$ is chosen accordingly to be a cokernel matrix; see \cite[Theorem 4.1.3]{coxlittleschneck}.

\paragraph{The Cox ring and the $\sigma$-positive property.} 
The homogeneous coordinate ring of a projective toric variety $X_{\Sigma}$, also known as the \textit{Cox ring}, is the ring $R = \mathbf{k}[x_{\rho}, \, \rho \in \Sigma(1)]$ which is $\Cl(X_{\Sigma})$-graded by means of the map $\pi$ defined in \eqref{eq:ses}, i.e.
$$R = \oplus_{\alpha \in \Cl(X_{\Sigma})} R_{\alpha},$$ with $R_{\alpha} = H^0(X_{\Sigma},\mathcal{O}_{\Sigma}(D))$ where $D$ is a torus-invariant Weil divisor such that $[D] = \alpha$ and $\mathcal{O}_{\Sigma}$ is the structure sheaf of $X_{\Sigma}$; see \cite{coxring} for more details.

In what follows, we will use the following notation for the variables of the Cox ring: assuming that there exists a maximal smooth cone $\sigma \in \Sigma(n)$, we will denote by $x_1,\dots,x_n$ the variables associated to the rays $\rho \in \sigma(1)$ and by $z_1,\dots,z_r$ the remaining variables of $R$. Denote by $u_1,\dots,u_n,u_{n+1},\dots,u_{n+r}$ the generators of the rays associated to $x_1,\dots,x_n,z_1,\dots,z_r$, respectively. According to this choice of $\sigma$, one can always write a matrix  of the map $\pi$ in \eqref{eq:ses} under the form
\begin{equation}\label{eq:pimatrix}
\pi = \begin{pmatrix}
  \mathcal{P} & \Id_r 
\end{pmatrix},
\end{equation}
where $\mathcal{P}$ is a block matrix $(\mathcal{P}_{j,k})_{1\leq j\leq r, 1\leq k\leq n}$
whose rows correspond to the relations between $u_{n+j}$ and the basis given by $u_1,\dots,u_n$ for $j = 1,\dots,r$. In other words, each row of $\pi$ corresponds to a relation of the form: 
\begin{equation}
\label{equationu}
    u_{n+j} + \sum_{k = 1}^n\mathcal{P}_{j,k}u_{k} = 0 \quad j = 1,\dots,r.
\end{equation}
In order to introduce Sylvester forms later on, we need the following property which is not standard. 

\begin{definition}
\label{def:definitionpositive} For $\sigma \in \Sigma(n)$, the projective toric variety 
$X_{\Sigma}$ is called \textit{$\sigma$-positive} if $\sigma$ is a maximal smooth cone such that a matrix of the map $\pi$ defined in \eqref{eq:ses} can be written as in \eqref{eq:pimatrix} with the additional condition that 
$\mathcal{P}_{j,k} \geq 0$ for $j = 1,\dots,r$ and $k = 1,\dots,n$. \ccc{This property is equivalent to require that the vectors $-u_{n+j}$ belong to $\sigma$ for all $j = 1,\dots,r$ (see Figure \ref{fig:sigmapositive} for an illustrative example).}  
\end{definition}

\begin{figure}[h]
    \centering
    \begin{tikzpicture}

\fill[color=lightgray] (0,0) rectangle (1,1) ;

    \draw[red,dashed] (0,0) -- (1,1);

\draw[thick,blue] (0,0) -- (1,0);
\draw[thick,blue] (0,0) -- (0,1);
\draw[thick,blue] (0,0) -- (-1,-1);

\end{tikzpicture}
    \caption{\ccc{The complete fan whose rays are generated by the vectors $(1,0),(0,1),(-1,-1)$ (in \textcolor{blue}{blue}) has the $\sigma$-positive property with respect to $\sigma = \langle (1,0),(0,1) \rangle$ (in \textcolor{gray}{gray}) because the only ray that does not belong to $\sigma$ satisfies that $-(-1,-1) \in \sigma$ (in \textcolor{red}{red}).}}
    \label{fig:sigmapositive}
\end{figure}

A first observation is that not all smooth toric varieties are $\sigma$-positive for some $\sigma \in \Sigma(n)$, as shown in the following example.

\begin{example}
Let $\Sigma$ be the complete smooth fan in $N_{\mathbb{R}} = \mathbb{R}^2$ with the following rays:
\begin{equation*}
  \rho_1 = (1,0) \: \rho_2 = (0,1) \: \rho_3 = (-1,1) \: \rho_4 = (-1,0) \: \rho_5 = (-1,-1) \: \rho_6 = (0,-1).  
\end{equation*}
It is straightforward to check that for every $\sigma \in \Sigma(2)$, there is $\rho \notin \sigma(1)$ such that $-u_{\rho} \notin \sigma$.
\end{example}

On the other hand, \ccc{many} of the projective toric varieties that are of interest in applications are $\sigma$-positive for some smooth maximal cone $\sigma$. For instance, this property is preserved under the product of toric varieties. To be more precise, recall that the product of two toric varieties is defined by the product fan; see \cite[Theorem 2.4.7]{coxlittleschneck}. Any cone of this fan is of the form $\sigma_1 \times \sigma_2$, where its elements are considered as pairs $(u,v)$ for $u \in \sigma_1$ and $v \in \sigma_2$. Moreover, $\dim \sigma_1 \times \sigma_2 = \dim \sigma_1 + \dim \sigma_2$.

\begin{lemma}\label{lem:productpositivity}  If $X_{1}$ (resp.~$X_2$) is a toric variety which is $\sigma_1$-positive (resp.~$\sigma_2$-positive) for some maximal cone $\sigma_1$ in a fan $\Sigma_1$, (resp.~$\sigma_2$ in a fan $\Sigma_2$), then the product $X_1 \times X_2$ is 
$(\sigma_1 \times \sigma_2)$-positive. 
\end{lemma}

\begin{proof}
Any ray $\rho$ of the product fan is generated by an element of the form $(u_{\rho_1}, 0)$ or $(0, u_{\rho_2})$, where $\rho_1$ is a ray of $\sigma_1$ and $\rho_2$ is a ray of $\sigma_2$. By assumption, $-u_{\rho_1}$ and $-u_{\rho_2}$ can be written as a positive combination of elements in either $\sigma_1$ or $\sigma_2$; therefore, they belong to $\sigma_1 \times \sigma_2$.
\end{proof}

\begin{example}
\label{example:h1}
The projective space $\mathbb{P}^n$ is $\sigma$-positive as the map $\pi$ can be written as $\pi = ( 1 \cdots 1 )$ for any choice of the maximal cone $\sigma$. Therefore, any product of projective spaces is $\sigma$-positive by Lemma \ref{lem:productpositivity}. Another classical family of smooth toric varieties are Hirzebruch surfaces $\mathcal{H}_{\ccc{b}} \subset \mathbb{R}^2$: for each \ccc{$b \in \mathbb{Z}_{>0}$}, these varieties correspond to the fans \ccc{$\Sigma_b$} with rays
\[\rho_1 = (1,0) \: \rho_2 = (0,1) \: \rho_3 = (-1,-\ccc{b}) \: \rho_4 = (0,-1). \]
Hirzebruch surfaces are smooth and $\sigma$-positive with respect to the smooth maximal cone $\sigma = \langle \rho_1, \rho_2 \rangle$ as $\pi$ can be written as
\[
\pi = \begin{pmatrix}
  1 & \ccc{b} & 1 & 0 \\ 
  0 & 1 & 0 & 1
\end{pmatrix}.
\] 
\end{example}

\paragraph{Generic homogeneous sparse polynomial systems.} 
Let $\Delta_0,\dots,\Delta_n$ be lattice polytopes in $M_{\mathbb{R}}$, let $\Sigma$ be the normal fan of the Minkowski sum $\Delta = \sum_{i = 0}^{n} \Delta_i$ and let $X_{\Sigma}$ be the corresponding toric variety. Assume that $\Delta$ is $n$-dimensional. As $X_{\Sigma}$ is defined by the normal fan of $\Delta$, $X_{\Sigma}$ is complete and $\Delta$ corresponds to an ample divisor \cite[Proposition 6.1.4]{coxlittleschneck}, which implies that $X_{\Sigma}$ is projective.

Suppose that $X_{\Sigma}$ admits a maximal smooth cone $\sigma \in \Sigma(n)$ and label  the rays $(\rho_j)_{j = 1,\dots,n+r}$ and their generators $(u_j)_{j = 1,\dots,n+r}$ according to the order of the variables $x_1,\dots,x_n,z_1,\dots,z_r$ as 
 in \eqref{eq:pimatrix}.  
The lattice polytopes $\Delta_i$ can be identified with elements $a_i = (a_{i,j})_{j = 1,\dots,n+r} \in \mathbb{Z}^{\Sigma(1)} = \mathbb{Z}^{n+r}$ using the following facet presentations
\begin{equation}\label{eq:polytopes}
    \Delta_i = \{m \in M_{\mathbb{R}} \, : \, \langle u_j,m \rangle \geq -a_{i,j}, \, j = 1,\dots,n+r\}, \, i = 0,\ldots,n.
\end{equation}
{These polytopes correspond to divisors $\sum_{j} a_{i,j} D_j$ where $D_j$ is the torus invariant divisor associated with the ray $\rho_j$ for $j = 1,\dots,n+r$. 
\begin{remark}
    We recall that a Cartier divisor $D$ is nef, if and only if, $D$ is generated by global sections; and ample, if and only if, the normal fan of its associated polytope is $\Sigma$; see \cite[Theorem 6.3.12, Theorem 6.1.14, Proposition 7.2.3]{coxlittleschneck}.
\end{remark}
In particular, for all $i \in \{0,\dots,n\}$, we can choose $-a_{i,j}$ to be $\min_{m \in \Delta_i} \langle u_j,m \rangle $ (which we can always assume to be an integer as $\Delta_i$ is a lattice polytope), the polytopes $\Delta_0,\dots,\Delta_n$ correspond to nef Cartier divisors (see \cite[Proposition 6.1.1]{coxlittleschneck}). By \cite[Proposition 4.2.8]{coxlittleschneck}, this assumption also implies the following property:
\begin{equation}
\label{equationcartier}
\forall i \in \{ 0,\dots,n\} \text{ and }  \forall \tau \in \Sigma(n), \text{ there is } m_{i,\tau} \in \mathcal{A}_i \text{ such that for } \rho_j \in \tau(1) \quad \langle u_j, m_{i,\tau}\rangle = -a_{i,j}.  
\end{equation} 
Using the short exact sequence \eqref{eq:ses}, we observe that two polytopes whose vectors in $\mathbb{Z}^{n+r}$ map to the same class in $\Cl(X_{\Sigma})$ are translations of each other; see \cite[§4.2, §4.3]{coxlittleschneck}. Moreover, we also note that each vertex of the polytopes $\Delta_0,\ldots,\Delta_n$ corresponds to a maximal cone in a fan that is refined by $\Sigma$; see \cite[Proposition 6.2.5]{coxlittleschneck}. Thus, for each $i = 0,\dots,n$, we choose the presentation of the polytope $\Delta_i$ so that the lattice point in \eqref{equationcartier} associated to the cone $\sigma$ (corresponding to the variables $x_1,\dots,x_n$) is $0 \in M$. In particular, we can see that this choice implies that $a_{i,k} = 0$ for all $k = 1,\dots,n$. Hence, for each $i = 0,\dots,n$, we are choosing a representative of the class of polytopes of $\Delta_i$ that only depends on $a_{i,n+1},\dots,a_{i,n+r}$. We write this representative as $(a_{i,n+j})_{j = 1,\dots,r} \in \mathbb{Z}^{r}$. For instance, if  $X_{\Sigma} = \mathbb{P}^n$ and $\sigma$ is the cone generated by the canonical basis of $\mathbb{Z}^n$, the $\Delta_i$'s  \eqref{eq:polytopes} only depend on a positive integer $a \in \mathbb{Z}_{>0}$ and one recovers the Newton polytopes of the polynomials of degree $a$.

More generally, whenever we refer to a polytope $\Delta_{\nu}$ associated with any other nef Cartier class $\nu \in \Cl(X_{\Sigma})$, we also write:
\begin{equation}\label{eq:polytopesnu}
    \Delta_{\nu} = \{m \in M_{\mathbb{R}} \, : \, \langle u_j,m \rangle \geq -\nu_{j}, \, j = 1,\dots,n+r\},
\end{equation}
and this polytope only depends on $(\nu_{n+j})_{j = 1,\dots,r} \in \mathbb{Z}^{r}$ 
as $\nu_k = 0$ for $k = 1,\dots,n$. In particular, the way we wrote the map $\pi$ in \eqref{eq:pimatrix} implies that every monomial $x^{\mu} = x_1^{\mu_1}\cdots x_n^{\mu_n}z_1^{\mu_{n+1}}\cdots z_r^{\mu_{n+r}}$ of degree $\nu$ is mapped via $\pi$ to $(\nu_{n+j})_{j = 1,\dots,r} \in \mathbb{Z}^r$ and thus satisfies the relations:
\begin{equation}
    \label{eq:murelation}
    \nu_{n+j} = \mu_{n+j} + \sum_{k = 1}^n\mathcal{P}_{j,k}\mu_k \text{ for all } j  = 1,\dots,r.
\end{equation}}

\begin{remark}
\label{remarknegative}
  Writing the polytopes in the presentation \eqref{eq:polytopes} and imposing that $0 \in \Delta_i$ also implies that for any $\nu \in \Cl(X_{\Sigma})$, we have $\nu_{n+j} \geq 0$ for $j = 1,\dots,r$. Otherwise, $0 = \langle u_{n+j}, 0 \rangle \geq -\nu_{n+j} > 0$. In particular, if $\nu_{n+j} < 0$ for some $j \in \{1,\dots,r\}$, then there are no monomials in $R$ of degree $\nu$.
\end{remark}

Let $\alpha_0,\dots,\alpha_n$ be nef Cartier classes in $\Cl(X_{\Sigma})$ associated to $\Delta_0,\dots,\Delta_n$, and $R_{\alpha_0},\dots,R_{\alpha_n}$ be the corresponding graded components in the Cox ring, respectively. These graded components are finite $\mathbf{k}$-vector spaces and have a monomial basis given by monomials $x^{\mu} := x_1^{\mu_1}\cdots x_n^{\mu_n}z_1^{\mu_{n+1}}\cdots z_r^{\mu_{n+r}} \in R$. Let $A = \mathbf{k}[c_{i,\mu} \, : \, x^{\mu} \in R_{\alpha_i}, \, i = 0,\dots,n]$ and $C = A[x_1,\dots,x_n,z_1,\dots,z_r]$. The \textit{generic homogeneous sparse polynomial system of degrees $\alpha_0,\ldots,\alpha_n$} is the system defined by the polynomials 
\begin{equation}\label{eq:Fi}
F_i = \sum_{x^{\mu} \in R_{\alpha_i}}c_{i,\mu}x^\mu \in C, \ \ i=0,\ldots,n.
\end{equation} 
The ring $C$ can be interpreted as the Cox ring of the toric variety $X_{\Sigma}\times_{\mathbf{k}} \Spec(A)$ over generic coefficients and its graded components are given by $C_\alpha=R_\alpha\otimes_{\mathbf{k}} A$ for $\alpha \in \Cl(X_{\Sigma})$.

If the system is dehomogenized by setting $z_1 = \dots = z_r = 1$, we denote by $\Tilde{F}_i$ the polynomial obtained by substituting all $z_j$'s by 1 in $F_i$ for all $i = 0,\dots,n$. \ccc{After applying the change of coordinates defined by the basis of $\sigma$, the Newton polytope of $\Tilde{F}_i$ is $\Delta_i$ (we notice that this change of variables leaves invariant the sparse resultants that we will consider in Section \ref{sec:resultants}).} Conversely, the polynomials
 $F_0,\dots,F_n$ can be defined as the homogenizations of the polynomials  $\Tilde{F}_0, \ldots , \Tilde{F}_n$ with supports in the subsets $\mathcal{A}_i = \Delta_i \cap M$, $i = 0,\dots,n$. More precisely, the homogenization of the polynomial 
\begin{equation}
    \label{nothomogenized}\Tilde{F}_i = \sum_{m \in \mathcal{A}_i}c_{i,m}x^m \in \Tilde{C} = A[x_1,\dots,x_n]
\end{equation}
is the polynomial 
\begin{equation}
    \label{eqhomogenized}{F}_i = \sum_{m \in \mathcal{A}_i}c_{i,m}x^{\mathbf{F}m + a_i} \in {C} = A[x_1,\dots,x_n,z_1,\dots,z_r]
\end{equation}
where $\mathbf{F}$ and the $a_i$'s are defined in \eqref{eq:ses} and \eqref{eq:polytopes}, respectively. We note that  we can choose a monomial basis of $R_{\alpha_i}$ corresponding to $x^{\mu}$ where $\mu = \mathbf{F}m + a_i$ for each $m \in \mathcal{A}_i$. We refer the reader to \cite[Section 2.2]{bender2021toric} for more details about homogenization and dehomogenization of sparse polynomial systems.

\paragraph{A decomposition property.}
In what follows, we prove the existence of certain decompositions of homogeneous polynomials that we will use in Section \ref{sec:sylvforms} for defining toric Sylvester forms. For the sake of clarity, we denote with a lowercase letter $f$ any polynomial in the Cox ring $R$ of a toric variety $X_{\Sigma}$, whose coefficient ring is a field, in contrast with generic polynomials that we denoted above with a capital letter (see also Notation \ref{notSec5}).

\begin{theorem}
\label{theoremdecompositionrestated}
Let $X_{\Sigma}$ be a projective toric variety of dimension $n$ such that $X_{\Sigma}$ is $\sigma$-positive with respect to a smooth cone $\sigma \in \Sigma(n)$. Let $J$ be an ideal of \lb{the Cox ring  $R$ of $X_\Sigma$,} generated by homogeneous polynomials $f_0,\ldots,f_n$ of degrees $\alpha_0,\ldots,\alpha_n$, respectively, whose polytopes $\Delta_0,\dots,\Delta_n$ are written as in \eqref{eq:polytopes} and only depend on $(a_{i,n+j})_{j = 1,\dots,r} \in \mathbb{Z}^{r}$. Let $\nu \in \Cl(X_{\Sigma})$ be a nef Cartier class and let $\Delta_{\nu}$ be the corresponding polytope, written as in \eqref{eq:polytopesnu}, for some $(\nu_{n+j})_{j = 1,\dots,r} \in \mathbb{Z}^{r}$ which \ccc{satisfy} \begin{equation}\label{eq:constraintnu}0 \leq \nu_{n+j}< \min_{\lb{i=0,\ldots,n}} a_{i,n+j} \text{ for all } j = 1,\dots,r \end{equation} 
Then, the two following properties hold:
\begin{itemize}
    \item[\lb{(i)}] $R_{\nu} = (R/J)_{\nu}$.
    \item[\lb{(ii)}] For every $x^\mu \in R_{\nu}$ and $f_i \in R_{\alpha_i}$ and $i = 0,\dots,n$, there exists a decomposition of the form
    \begin{equation}\label{eq:decomposition} f_i = z_1^{\mu_{n+1}+1}\cdots z_r^{\mu_{n+r}+1}f^{\mu}_{i,0} + x_{1}^{\mu_1+1}f^{\mu}_{i,1} + \dots + x_n^{\mu_n + 1}f^{\mu}_{i,n} \end{equation}
    where the \lb{$f_{i,j}^{\mu}$, $i,j = 0,\dots,n$}, are homogeneous polynomials in $R$.
    \end{itemize}
\end{theorem}
\begin{proof}
The graded quotient map $R_{\nu} \xrightarrow[]{} (R/J)_{\nu}$ is surjective. If there is a nonzero polynomial of degree $\nu$ in $J$, there must be a monomial $x^{\mu} \in R_{\nu}$ that is divided by some monomial $x^{\mu_i} \in R_{\alpha_i}$ of degree $\alpha_{i}$ for some $i  \in  \{0,\dots,n\}$, i.e.~the degrees of the generators of $J$. If that is the case, then $x^{\mu} = x^{\Tilde{\mu}}x^{\mu_{i}}$ for some monomial $x^{\Tilde{\mu}}$ of degree $\nu - \alpha_i \in \Cl(X_{\Sigma})$. However, using \eqref{eq:constraintnu}, we see that $\nu_{n+j} - a_{i,n+j} < 0$ and by Remark \ref{remarknegative}, there cannot be any monomials of this degree in $R$. Thus, the kernel of the previous map is zero, proving (i).

\lb{We turn to the proof of (ii).} Recall from \eqref{eqhomogenized} that every monomial of degree $\alpha_i$ (and thus every monomial in $f_i$) can be written as $x^{\mathbf{F}m + a_i}$ for some lattice point $m \in \mathcal{A}_i$. Thus, we fix a monomial $x^{\mathbf{F}m_0 + a_{i}}$ in the support of $f_i$ for some $m_0 \in \mathcal{A}_i$.

Given $x^\mu = x_1^{\mu_1}\dots x_n^{\mu_n}z_1^{\mu_{n+1}}\dots z_r^{\mu_{n+r}} \in R_{\nu}$, \lb{we are going to show} that if $x^{\mathbf{F}m_0 + a_i} \in R_{\alpha_i}$ is not divisible by the monomial $z_1^{\mu_{n+1}+1}\cdots z_r^{\mu_{n+r}+1}$, then it must be divisible by one of the monomials $x_1^{\mu_1 + 1}, \dots, x_n^{\mu_n + 1}$. \lb{Indeed, if this property holds,} every monomial in $f_i$ for $i = 0,\dots,n$ must be divided by either $x_1^{\mu_1 + 1}, \dots, x_n^{\mu_n + 1}$ or $z_1^{\mu_{n+1}+1}\cdots z_r^{\mu_{n+r}+1}$ \lb{and the decompositions \eqref{eq:decomposition} follow}.

Using that $a_{i,k} = 0$ for $k = 1,\dots,n$, the $n + r$ components of $\mathbf{F}m_0 + a_i$ are:
\begin{equation}
\label{eqhomo}
    \begin{cases}
    \langle u_k, m_0 \rangle & k \in \{1,\dots,n\} \\ \langle u_{n+j}, m_0 \rangle + a_{i,n+j} & j \in \{1,\dots,r\}.
    \end{cases}
\end{equation}
Thus, the fact that $x^{\mathbf{F}m_0 + a_i}$ is not divisible by $z_1^{\mu_{n+1}+1}\cdots z_r^{\mu_{n+r}+1}$ implies that
\[ \langle u_{n+j_0}, m_0 \rangle + a_{i,n+j_0} \leq \mu_{n+j_0} \text{ for some } j_0 \in \{1,\dots,r\}.\]
From here, using \eqref{eq:constraintnu}, we get that $\langle u_{n+j_0}, m_0 \rangle + \nu_{n+j_0} < \mu_{n+j_0}$. On the other hand, \lb{the monomial} $x^{\mu} = x_1^{\mu_1}\dots x_n^{\mu_n}z_1^{\mu_{n+1}}\dots z_r^{\mu_{n+r}}$ is of degree $\nu$ and \lb{hence} by \eqref{eq:murelation}, we have $\nu_{n+j_0} = \mu_{n+j_0} + \sum_{k = 1}^n\mathcal{P}_{j_0,k}\mu_{k}$, \lb{which implies} that
$$
    \langle u_{n+j_0}, m_0 \rangle + \sum_{k = 1}^n\mathcal{P}_{j_0,k}\mu_k < 0. $$
Finally, we use the relation \eqref{equationu} between $u_{n + j_0}$ and the generators of $\sigma$ to derive \lb{the inequality}
$$ \sum_{k = 1}^n\mathcal{P}_{j_0,k}(\mu_k - \langle u_k, m_0 \rangle) < 0.$$
As $X_{\Sigma}$ is $\sigma$-positive and all the \lb{$\mathcal{P}_{j_0,k}$'s} are non-negative integers for all $k = 1,\dots,n$ (and they are not all equal to $0$ as $u_{n + j} \neq 0$), \lb{there exists} $k_0 \in \{1,\dots,n\}$ such that $\mu_{k_0} - \langle u_{k_0}, m \rangle < 0$. \lb{As the exponent of $x_{k_0}$ in $x^{\mathbf{F}m_0 + a_i}$ is precisely $\langle u_{k_0},m_0 \rangle$, we deduce that $x_{k_0}^{\mu_{k_0} + 1}$ divides $x^{\mathbf{F}m_0 + a_i}$}.
\end{proof}

\begin{corollary}
\label{corollarydecomposition}
Assume that the projective toric variety 
$X_{\Sigma}$ is $\sigma$-positive for some $\sigma \in \Sigma(n)$. If the polytopes $\Delta_i$ in \eqref{eq:polytopes}  are $n$-dimensional for all $i = 0,\dots,n$, then Theorem \ref{theoremdecompositionrestated} holds for $(\nu_{n+j})_{j = 1,\dots,r} = 0 \in \mathbb{Z}^r$.
\end{corollary}

\begin{proof}
If there are $i_0 \in \{0,\dots,n\}$ and $j_0 \in \{1,\dots,r\}$ such that $a_{i_0,n+j_0} = 0$, then for every $m \in \mathcal{A}_{i_0} = \Delta_{i_0} \cap M$, we have the inequality $ \langle u_{n+j_0}, m \rangle \geq 0$. Using the relation \eqref{equationu}, we get $$\sum_{k = 1}^n\mathcal{P}_{j_0,k} \langle u_{k}, m \rangle \leq 0 \quad \forall m \in \mathcal{A}_{i_0}.$$
As $X_{\Sigma}$ is $\sigma$-positive and the $\mathcal{P}_{j_0,k}$ are non-negative integers for $k = 1,\dots,n$ (not all equal to zero as $u_{n + j_0} \neq 0$), there must be some $k_0 \in \{1,\dots,n\}$ such that $\langle u_{k_0}, m \rangle \leq 0$ for all $m \in \mathcal{A}_{i_0}$. On the other hand, we also have the inequality $\langle u_{k_0}, m \rangle \geq 0$ for all $m \in \mathcal{A}_{i_0}$ due to the presentation in \eqref{eq:polytopes} and using that $a_{i_0,k} = 0$ for all $k = 1,\dots,n$. Thus, the lattice points in $\Delta_{i_0}$ must satisfy $\langle u_{k_0}, m \rangle = 0$ and thus $\Delta_{i_0}$ cannot be $n$-dimensional.

Therefore, if the $\Delta_i$ are $n$-dimensional for all $i = 0,\dots,n$, we have $0 < \min_{\lb{i=0,\ldots,n}} a_{i,n+{j}}$ for $j = 1,\dots,r$, which proves that $(\nu_{n+j})_{j = 1,\dots,r} = 0 \in \mathbb{Z}^r$ satisfies the hypotheses of Theorem \ref{theoremdecompositionrestated}.
\end{proof}

Finally, we note that that if $X_\Sigma$ is assumed to be $\sigma$-positive, then Theorem \ref{theoremdecompositionrestated} can be easily extended to the setting of generic homogeneous sparse polynomials in \eqref{eq:Fi} and yield a decomposition of the $F_i$ for $i=0,\ldots,n$, over $X_{\Sigma} \times_{\mathbf{k}} \Spec(A)$.

\begin{remark}
    If $X_{\Sigma}$ does not have the $\sigma$-positive property, but one can  find another way to decompose the polynomials $F_i$ for $i = 0,\dots,n$ as in \eqref{eq:decomposition}, then the \lb{results presented in the next sections hold similarly}. One \lb{such} example is the construction of the form $\Delta_{\sigma}$ with a nonzero residue, as detailed in \cite[Theorem 0.2]{cattani1995residues}, which relies on the polynomials $F_i$ corresponding to $\mathbb{Q}$-ample divisors.
\end{remark}

\paragraph{Torsion and local cohomology.}
From the fan $\Sigma$ of a \lb{toric variety $X_\Sigma$}, the irrelevant ideal $\mathfrak{b}$ of its homogeneous coordinate ring $R=\mathbf{k}[x_{\rho}, \, \rho \in \Sigma(1)]$ is defined as
\begin{equation}
\label{irrelevant}
    \mathfrak{b} = ( \tilde{x}^{\tau} \textrm{ such that } \tau \in \Sigma(n) ), \textrm{ where } \tilde{x}^{\tau} = \prod_{\rho \notin \tau(1)}x_{\rho}.
\end{equation}
The $\mathfrak{b}$-torsion of a graded $R$-module $S$ is classically defined as   
\[ \Gamma_{\mathfrak{b}}(S) = \{a \in S \, : \, \exists k \in \mathbb{Z}_{>0} \: \mathfrak{b}^k·a = 0\}\]
and the local cohomology modules $H^i_{\mathfrak{b}}(S)$ are then the derived functors of $S \xrightarrow[]{} \Gamma_{\mathfrak{b}}(S)$. When the module $S$ is a quotient ring $B = C/I$ for $I = \langle F_0,\dots,F_n \rangle$ the ideal generated by the polynomials defined in \eqref{eq:Fi}, the $0$-th local cohomology is $H_{\mathfrak{b}}^0(B) = I^{\sat}/I$ where $I^{\sat}$ denotes the saturation of $I$ with respect to the irrelevant ideal, i.e.~$ I^{\sat} := (I:\mathfrak{b}^{\infty}) = \{ p \in C \, : \,  \exists k \in \mathbb{Z}_{>0} \quad \mathfrak{b}^k p \subset I \}$. 

Local cohomology modules are strongly related to sheaf cohomology modules. More precisely, let $S$ be a finitely generated $\Cl(X_{\Sigma})$-graded $R$-module with associated coherent sheaf $\mathcal{S}$ in $X_{\Sigma}$ and $\alpha \in \Cl(X_{\Sigma})$. If $p \geq 2$, then
\begin{equation}\label{localsheaf} H_{\mathfrak{b}}^p(S)_{\alpha} \simeq H^{p-1}(X_{\Sigma},\mathcal{S}(\alpha)),\end{equation}
where $H^p_{\mathfrak{b}}(S)_{\alpha}$ is the graded piece of $H^p_{\mathfrak{b}}(S)$ of degree $\alpha$ and $\mathcal{S}(\alpha)$ is the sheaf defined by $\mathcal{S} \otimes \mathcal{O}_{\Sigma}(D)$, for a divisor $D$ with $[D] = \alpha$ and $\mathcal{O}_{\Sigma}$ the structure sheaf of $X_{\Sigma}$. Furthermore, the following exact sequence holds (see \cite[Theorem 9.5.7]{coxlittleschneck} for proofs):
\[ 0 \xrightarrow{} H_{\mathfrak{b}}^0(S)_{\alpha} \xrightarrow[]{} S_{\alpha} \xrightarrow{} H^{0}(X_{\Sigma},\mathcal{S}(\alpha)) \xrightarrow[]{} H_{\mathfrak{b}}^1(S)_{\alpha} \xrightarrow[]{} 0.\]
If $S = R$, then $R_{\alpha} = H^0(X_{\Sigma}, \mathcal{O}_{\Sigma}(\alpha))$ and therefore 
\begin{equation}
    \label{remarkglobalsections}
    H^0_{\mathfrak{b}}(R) = H^1_{\mathfrak{b}}(R) = 0,
\end{equation}
which implies that $H^{i}_{\mathfrak{b}}(C) = 0$ for $i = 0,1$. 

\begin{notation} For the sake of simplicity in the notation, 
for any Cartier divisor $D$ and any integer $p\geq 0$, we will write $H^p(X_{\Sigma},\alpha)$ in place of $H^p(X_{\Sigma},\mathcal{O}_{\Sigma}(D))$, where 
$\alpha = [D] \in \Cl(X_{\Sigma})$. 
\end{notation}

\lb{The following theorems, that are originally due to Demazure and Batyrev-Borisov, will be our main tools to analyze the vanishing of sheaf cohomology modules over toric varieties (see \cite[Theorem 9.2.3]{coxlittleschneck} and \cite[Theorem 9.2.7]{coxlittleschneck} for proofs).
\begin{theorem}[Demazure]\label{demazure}
Let $X_{\Sigma}$ be a toric variety such that $\Sigma$ is complete and $D$ be a nef Cartier divisor, then  
$H^p(X_{\Sigma}, \alpha) \simeq 0$ for all $p > 0$ and $\alpha = [D]$.
\end{theorem}

\begin{theorem}[Batyrev-Borisov]\label{batyrevborisov}
    Let $X_{\Sigma}$ be a  complete toric variety and $D$ be a nef Cartier divisor, then 
$$H^p(X_{\Sigma},-\alpha)  \simeq 
\begin{cases}
0 & \textrm{ if } p \neq \dim \Delta_{\alpha} \\ 
\oplus_{m \in \Relint(\Delta_{\alpha}) \cap M}K\chi^{-m} & \textrm{ if } p = \dim \Delta_{\alpha} 
\end{cases}$$
where $\alpha = [D] \in \Cl(X_{\Sigma})$ and $\Relint(\Delta_{\alpha})$ denotes the relative interior of the polytope $\Delta_{\alpha}$ associated with $\alpha$.
\end{theorem}

\begin{remark}
We notice that the two above theorems are proved in more generality in \cite{coxlittleschneck}, we stated them with assumptions that are sufficient in our context.
\end{remark} 

 Another important result we will use is the toric version of Serre duality (see \cite[Theorem 9.2.10]{coxlittleschneck} for a proof): for any Cartier divisor $D$ and any integer $p \geq 0$, 
\begin{equation}\label{serreduality} 
    H^p(X_{\Sigma}, \alpha) \cong H^{n-p}(X_{\Sigma}, -K_X-\alpha)^\vee,
\end{equation} 
where $K_{X}$ is the anticanonical class in $\Cl(X_{\Sigma})$ and $\alpha = [D] \in \Cl(X_{\Sigma})$.
}

\paragraph{Hilbert functions and the Grothendieck-Serre formula.} Let $X_{\Sigma}$ be a projective toric variety and let $R$ be its Cox ring. The Hilbert function of a finitely generated graded $R$-module $S$ is defined \lb{by}
\begin{eqnarray*}
\HF(S,-): \Cl(X_{\Sigma}) & \rightarrow & \mathbb{Z}_{\geq 0} \\ 
\alpha & \mapsto & \HF(S,\alpha) := \dim_{\mathbf{k}}(S_{\alpha}).	
\end{eqnarray*}
Assuming that $X_\Sigma$ is a \textit{smooth} toric variety, then for $\alpha \gg 0$ (component-wise), this function becomes a (multivariate) polynomial called the Hilbert polynomial and is denoted by $\HP(S,\alpha)$; see \cite[Lemma 2.8]{maclagan2003uniform}.
\begin{remark}
\label{remarkzerodimensional}
If $S = R/J$ with $J$  \lb{a homogeneous} ideal of $R$ defining a $0$-dimensional subscheme in $X_{\Sigma}$, then the Hilbert polynomial of $S$ is a constant which is equal to the number of points (over an algebraic closure of $\mathbf{k}$) in this subscheme, counted with multiplicity.
\end{remark} 
An important relation between the Hilbert function, the Hilbert polynomial and local cohomology modules is given by the Grothendieck-Serre formula (see \cite[Proposition 2.14]{maclagan2003uniform} for a proof): for any $\alpha \in \Cl(X_{\Sigma})$,
\begin{equation}
\label{grserre}
    \HF(S,\alpha) = \HP(S,\alpha) + \sum_{i = 0}^{n}(-1)^i\dim_{\mathbf{k}} H_\mathfrak{b}^i(S)_{\alpha}.
\end{equation} 

We notice that the smoothness assumption we will require on the toric variety $X_{\Sigma}$ in Sections \ref{sec::hybridmat}, \ref{sec:resultants} and \ref{sec:residues} is precisely due to the use of this formula.

\section{A duality theorem}\label{sec:duality}

Let $X_\Sigma$ be a projective toric variety of dimension $n$ \lb{which admits} a maximal smooth cone $\sigma \in \Sigma(n)$. 
In this section, we consider the ideal generated by $n+1$ generic homogeneous sparse polynomials (see Section \ref{sec:toricgeom}) and analyze some graded components of its saturation via a duality property. For that purpose, we take again the notation \eqref{eq:Fi}: $F_0,\dots,F_n$ are the generic homogeneous polynomials of degree $\alpha_0,\dots,\alpha_n$, respectively; they are of the form
\begin{equation}
F_i = \sum_{x^{\mu} \in R_{\alpha_i}}c_{i,\mu}x^{\mu} \in C = A[x_1,\dots,x_n,z_1,\dots,z_r].
\end{equation}

As a preliminary result, we first show that $F_0,\dots,F_n$ form a regular sequence outside $V(\mathfrak{b})\subset \Spec(C)$.

\begin{lemma}
\label{lemmaattheend}
For every maximal cone $\tau \in \Sigma(n)$ and for every $i = 0,\dots,n$, there is a lattice point $m_{i,\tau} \in \mathcal{A}_i$ and $L \in \mathbb{Z}_{> 0}$ such that $x^{\mathbf{F}m_{i,\tau} + a_i}$ divides $(\Tilde{x}^{\tau})^{L}$ where $\Tilde{x}^{\tau}$ is 
defined in \eqref{irrelevant}. 
\end{lemma}
\begin{proof}
    The exponents of of $x^{\mathbf{F}m + a_i}$ are:
    \begin{equation}
    \begin{cases}
    \langle u_k, m \rangle & k \in \{1,\dots,n\} \\ \langle u_{n+j}, m \rangle + a_{i,n+j} & j \in \{1,\dots,r\}.
    \end{cases}
\end{equation}
Thus, using \eqref{equationcartier}, we can find $m_{i,\tau} \in \mathcal{A}_i$ such that for $\rho_j \in \tau(1)$, we have $\langle u_j, m_{i,\tau} \rangle + a_{i,j} = 0$. Moreover, we can choose $L$ that bounds above $\langle u_j, m_{i,\tau} \rangle + a_{i,j}$ for $\rho_j \notin \tau(1)$. Therefore, $x^{\mathbf{F}m_{i,\tau} + a_i}$ divides $(\Tilde{x}^{\tau})^{L}$.
\end{proof}

\begin{lemma}
\label{lemmaregularsequence} The homogeneous generic polynomials 
$F_0,\dots,F_n$ define a regular sequence in the localization ring \lb{$C_{\Tilde{x}^{\tau}}$ for any $\tau \in \Sigma(n)$, where $\Tilde{x}^{\tau}$ is defined in \eqref{irrelevant}}. 

\end{lemma}

\begin{proof}
We claim that $F_0$ is a nonzero divisor in $C$. This follows \lb{from} Dedekind-Mertens Lemma \cite[Corollary 2.8]{busejouanolou}, which says that a polynomial $F$ is a nonzero divisor in $A[x_1,\dots,x_n]$ if its content ideal is a nonzero divisor in $A$. The content ideal is generated by the coefficients $c_{0,\mu}$ for $x^\mu \in R_{\alpha_0}$ and they are all nonzero divisors. Therefore, $F_0$ is a nonzero divisor also in $C_{\Tilde{x}^{\tau}}$ for all $\tau \in \Sigma(n)$. 

By Lemma \ref{lemmaattheend}, we can always find $m_{i,\tau} \in \mathcal{A}_i$ such that $x^{\mathbf{F}m_{i,\tau} + a_i}$ is invertible in the localization ring $C_{\Tilde{x}^{\tau}}$ and let $c_{i,\tau}$ be the coefficient in $A$ associated to this monomial. Then, similarly to \cite[Lemma 3.2]{buse2021multigraded}, for any $t\in \{0,\ldots,n-1\}$ there is an isomorphism of $(A_{\tau}^{t}[x_1,\dots,x_n,z_1,\dots,z_r])$-algebras
\[ \big(A[x_1,\dots,x_n,z_1,\dots,z_r]/\langle F_0,\dots,F_t \rangle\big)_{\Tilde{x}^{\tau}} \xrightarrow{\sim} (A_{\tau}^{t}[x_1,\dots,x_n,z_1,\dots,z_r])_{\Tilde{x}^{\tau}}  \]
where $A_{\tau}^t = \mathbf{k}[c_{i,\mu} \quad c_{i,\mu} \neq c_{i,\tau} \quad 0 \leq i \leq t]$ i.e., $A = A_{\tau}^t[c_{i,\tau} \quad 0 \leq i \leq t]$. This map sends $c_{i,\tau}$ to $\frac{-F_i + c_{i,\tau}x^{\mathbf{F}m_{i,\tau} + a_i}}{x^{\mathbf{F}m_{i,\tau} + a_i}}$ for $i = 0,\dots,t$, and leaves  the rest of coefficients and variables invariant. Applying again the Dedekind-Mertens Lemma as above, we deduce that the polynomial $F_{t+1}$ is a nonzero divisor in $(A_{\tau}^{t}[x_1,\dots,x_n,z_1,\dots,z_r])_{\Tilde{x}^{\tau}}$, and therefore in the localized quotient ring  $ \big(A[x_1,\dots,x_n,z_1,\dots,z_r]/\langle F_0,\dots,F_t \rangle\big)_{\Tilde{x}^{\tau}} $. 
\end{proof}
Next, we consider the two canonical spectral sequences associated with the \v{C}ech-Koszul double complex $\mathcal{C}_{\mathfrak{b}}^{\bullet}(K_{\bullet}(F))$, where $K_\bullet(F)$ denotes the Koszul complex of the sequence of homogeneous polynomials $F_0,\ldots,F_n$ in $C$. The terms of the Koszul complex are graded free $C$-modules and we denote their homology modules by $H_p$ for simplicity in the notation. If we start taking homologies horizontally, the second page is:
\[\begin{tikzpicture}
\matrix (m) [matrix of math nodes,
             nodes in empty cells,
             nodes={minimum width=5ex,
                    minimum height=5ex,
                    outer sep=-5pt},
             column sep=1ex, row sep=1ex,
             text centered,anchor=center]{
   H_{\mathfrak{b}}^{0}(H_{n+1})   &  H_{\mathfrak{b}}^{0}(H_{n})  &  H_{\mathfrak{b}}^{0}(H_{n-1}) & \cdots &  H^0_{\mathfrak{b}}(H_0)=I^{\sat}/I \\
   0   &  0  &  0 & \cdots &  H_{\mathfrak{b}}^{1}(H_{0}) \\ \vdots   &  \vdots         & \vdots          &  & \vdots \\                   
   0   &  0  & 0 & \cdots &  H_{\mathfrak{b}}^{n}(H_{0}) \\
   0   &  0 &  0 & \cdots &  H_{\mathfrak{b}}^{n+1}(H_{0}) \\};
\end{tikzpicture}.
\]
The vanishing of the local cohomology modules $H_{\mathfrak{b}}^i(H_j)$ for $i>0$ and $j>0$ follows from \lb{Lemma \ref{lemmaregularsequence} which shows that the $F_i$'s form a regular sequence outside $V(\mathfrak{b})$}. In addition, we deduce that $H_p$ are geometrically supported on $V(\mathfrak{b})$ for all $p>0$ by a classical property of Koszul complexes, and hence that $H_{\mathfrak{b}}^{0}(H_{p})=H_{p}$ for all $p>0$. 

On the other hand, if we start taking homologies vertically, we obtain the following first page:
\[\begin{tikzpicture}\label{eq:ss2}
\matrix (m) [matrix of math nodes,
             nodes in empty cells,
             nodes={minimum width=5ex,
                    minimum height=5ex,
                    outer sep=-5pt},
             column sep=1ex, row sep=1ex,
             text centered,anchor=center]{
   0  &  &  0 &  &  0 & \cdots &  0 \\
 0  &  &  0 &  &  0 & \cdots &  0  \\
   \vdots &   &  \vdots     &     & \vdots          &  & \vdots \\
  H_{\mathfrak{b}}^{n}(C(-\sum_j \alpha_j)) & \rightarrow & H_{\mathfrak{b}}^{n}(\oplus_{k}C(- \sum_{j \neq k}\alpha_j))  & \rightarrow &  H_{\mathfrak{b}}^{n}(\oplus_{k,k'}C(- \sum_{j \neq k,k'}\alpha_j)) & \cdots &  H_{\mathfrak{b}}^{n}(C) \\
   H_{\mathfrak{b}}^{n+1}(C(-\sum_j \alpha_j)) & \rightarrow & H_{\mathfrak{b}}^{n+1}(\oplus_{k}C(- \sum_{j \neq k}\alpha_j))  & \rightarrow & H_{\mathfrak{b}}^{n+1}(\oplus_{k,k'}C(- \sum_{j \neq k,k'}\alpha_j)) & \cdots &  H_{\mathfrak{b}}^{n+1}(C) \\};
\end{tikzpicture}
\]
using that $K_j(F) = \bigoplus^{J \subset \{ 0,\dots,r \}}_{|J| = j}C(-\sum_{k \in J} \alpha_k)$. We note that the vanishing of the two first rows follows from \eqref{remarkglobalsections} and  the vanishing of $H_{\mathfrak{b}}^{p}(C)$ for all $p > n + 1$ is a consequence of Grothendieck's vanishing theorem \cite[Theorem 3.6.5]{grothendieckvanishing}. 

\begin{notation}
\label{notation}
The support $\Supp S$ of a graded module $S$ is the subset of $\nu \in \Cl(X_{\Sigma})$ such that $S_{\nu} \neq 0$. We denote by $\Gamma_1 $ the support of the modules on the main diagonal, except on the last row, and by $\Gamma_0 $ the support of the modules in the diagonal under $\Gamma_1 $, except on the last row again, i.e.
\begin{equation}
\label{supportscohomology}
    \Gamma_i = \Supp(\oplus_{p = 0}^{n}H_{\mathfrak{b}}^{p}(K_{p+i-1}(F))) \quad i = 0,1.
\end{equation}
In addition, we define $\Gamma_{\Res}$ to be the support of all the cohomology modules that are appearing above the diagonal in the first page of the second spectral sequence, i.e. $\Gamma_{\Res} = \Supp( {\oplus_{i < j}}H_{\mathfrak{b}}^{i}(K_{j}(F))$. Moreover, from now on, we denote by $\delta$ the divisor class $\alpha_0 + \dots + \alpha_n - K_X$ where $K_X$ denotes the anticanonical divisor of $X_{\Sigma}$. 
\end{notation}

\begin{remark}\label{rm:acyclicity}In the above analysis of the two spectral sequences associated to $F$, we proved that $K_\bullet(F)_\alpha$ is an acyclic complex of $A$-modules for all $\alpha \notin \Gamma_{\Res}$.
\end{remark}

The comparison of the two above spectral sequences leads to the following duality theorem. 

\begin{theorem}\label{dualitytheorem}
Let $X_{\Sigma}$ be a projective toric variety 
which admits a maximal smooth cone $\sigma \in \Sigma(n)$ 
and let $\nu \in \Cl(X_{\Sigma})$ be a nef Cartier divisor. If $\delta - \nu \notin  \Gamma_0  \cup \Gamma_1 $, then
\[ (I^{\sat}/I)_{\delta - \nu} \simeq \Hom_{A}((C/I)_{\nu},A). \]
\end{theorem}

\begin{proof} From the comparison of the two spectral sequences associated to the double complex $\mathcal{C}_{\mathfrak{b}}^{\bullet}(K_{\bullet}(F))$, for all $\nu \in \Cl(X_{\Sigma})$ such that $\delta - \nu \notin \Gamma_0 \cup \Gamma_1$, we get 
an isomorphism 
$$(I^{\sat}/I)_{\delta - \nu} \simeq \Ker\left(H_{\mathfrak{b}}^{n+1}(C(-\sum_{j}\alpha_j)) \xrightarrow[]{} H_{\mathfrak{b}}^{n+1}(\oplus_{k}C(-\sum_{j \neq k}\alpha_j))\right)_{\delta - \nu}.$$  
Moreover, using toric Serre duality \eqref{serreduality} and the relation between sheaf and local cohomology modules \eqref{localsheaf}, we obtain
\begin{equation*}
   H_{\mathfrak{b}}^{n+1}(C(-\sum_{j}\alpha_j))_{\delta - \nu} \simeq H^n(X_{\Sigma},-\nu -K_X) \simeq H^0(X_{\Sigma}, \nu)^{\vee} \simeq \Hom_{A}(C_{\nu}, A). 
\end{equation*}
By the same argument, we also have $H_{\mathfrak{b}}^{n+1}(\oplus_{k}C(-\sum_{j \neq k}\alpha_i))_{\delta - \nu} \simeq \Hom_{A}(I_{\nu}, A)$. 
Using the first isomorphism, we get the duality property.
\end{proof}

\begin{corollary}\label{cor:freemodule}
Let $X_{\Sigma}$ be a projective toric variety which admits a maximal smooth cone $\sigma \in \Sigma(n)$. Let $\Delta_0,\dots,\Delta_n$ be lattice polytopes as in \eqref{eq:polytopes} corresponding to the polynomials $F_0,\dots,F_n$. Let $\nu \in \Cl(X_{\Sigma})$ be a nef Cartier class and $\Delta_{\nu}$ be the corresponding polytope, written as in \eqref{eq:polytopesnu}, satisfying $0 \leq \nu_{n+j} < \min_{i = 0,\dots,n} a_{i,n+j}$ for $j = 1,\dots,r$. Assume also that $\delta - \nu \notin \Gamma_0 \cup \Gamma_1$. Then,
\begin{equation*} (I^{\sat}/I)_{\delta - \nu} \simeq \Hom_{A}(C_{\nu},A).
\end{equation*}
In particular, $(I^{\sat}/I)_{\delta-\nu}$ is a free $A$-module whose rank is equal to the rank of $C_\nu$, \lbb{equivalently $\HF(R,\nu)$}. 
\end{corollary} 
\begin{proof}
Using Theorem \ref{theoremdecompositionrestated} $i)$ (which does not require the $\sigma$-positive property),
we can derive that $(C/I)_{\nu}=C_{\nu}$.
\end{proof}
\begin{remark}
 We notice that the case $\nu=0$, which corresponds to the isomorphism $(I^{\sat}/I)_{\delta}\simeq A$, appears in \cite{cattani1997residuesresultants} in the case the polytopes $\Delta_0,\ldots,\Delta_n$ are scaled copies of the same ample polytope.
\end{remark}

To close this section, we prove that if we consider a proper subset of the polynomials that generate $I$, then the corresponding ideal must be saturated at $\delta$. We will need this property in the next section.
\begin{lemma}
\label{lemmastrictusbset}
    Assume that the polytopes $\Delta_0,\dots,\Delta_n$ are $n$-dimensional. Let $T$ be a proper subset of 
    $\{0,\dots,n\}$ and consider the ideal $I_T = (F_i, \, i \in T)$. Then, $(I_T^{\sat})_{\delta} = (I_T)_{\delta}$.
\end{lemma}

\begin{proof}
    Consider the cohomology groups $H_{\mathfrak{b}}^i(K_j(F_T))$ where $K_j(F_T)$ denotes the Koszul complex associated to $I_T$. Then, by \eqref{localsheaf},
    $$H_{\mathfrak{b}}^{i}(K_j(F_T))_{\delta} = \bigoplus^{J \subset T}_{|J| = j}H_{\mathfrak{b}}^{i}\big(C(-\sum_{k \in J} \alpha_k)\big)_{\delta} = \bigoplus^{J \subset T}_{|J| = j}H^{i-1}\big(X_{\Sigma},\sum_{k \notin J} \alpha_k - K_X\big).$$
    Using Serre duality \eqref{serreduality}, each of the summands is of the form:
    $$H^{i-1}(X_{\Sigma}, \sum_{k \notin J}\alpha_k - K_X) \simeq H^{n-i+1}(X_{\Sigma}, -\sum_{k \notin J}\alpha_k).$$
    As $J \subset T$ is a proper subset, $\sum_{k \notin J}\alpha_k$ is nef and its associated polytope is $n$-dimensional. Therefore, we can apply Theorem \ref{batyrevborisov}, implying that $H_{\mathfrak{b}}^{i}(K_j(F_T))_{\delta} = 0$ for all $i \geq 2$. If $i = 0,1$, we can use \eqref{remarkglobalsections}. Therefore, comparing the two spectral sequences of the Čech-Koszul double complex, we get $\big(I_T^{\sat}/I_T\big)_{\delta} = 0.$
\end{proof}

\section{Toric Sylvester forms}\label{sec:sylvforms}

We take again the notation of Section \ref{sec:duality}. As a consequence of Corollary \ref{cor:freemodule}, some graded components of $I^{\sat}/I$ are free $A$-modules and hence a natural question is to provide explicit $A$-bases for them. This is precisely the goal of this section. We first describe the graded component $(I^{\sat}/I)_\delta$, which essentially follows from \cite{cattani1995residues}. Then, we introduce Sylvester forms to deal with the  other cases. 
In what follows, we assume that the projective toric variety $X_\Sigma$ is $\sigma$-positive with respect to \lb{a} maximal smooth cone $\sigma \in \Sigma(n)$.

\medskip

\lb{Along the same lines as} \cite{cattani1995residues}, a nonzero element \lb{in} $(I^{\sat}/I)_{\delta}\simeq A$ \lb{can be constructed} as follows. Using Corollary \ref{corollarydecomposition}, if the polytopes $\Delta_0,\dots,\Delta_n$ are $n$-dimensional, one can decompose each polynomial as
\begin{equation}\label{eq:decompsylv}
F_i = z_1\cdots z_r {F}_{i,0} + x_1{F}_{i,1} + \dots + x_n{F}_{i,n},    
\end{equation}
and consider the determinant
\[\Sylv_0 = \det\begin{pmatrix} F_{i,j}\end{pmatrix}_{0\leq i,j \leq n}.\]
This homogeneous polynomial is called the \textit{toric jacobian}; we will denote its class modulo $I$ by $\sylv_0$. Observe that, by construction, $\Sylv_0$ is a linear form with respect to the coefficients of each $F_i$, $i = 0,\dots,n$.

\begin{lemma}
\label{allthevariables}
   Assume that $\Delta_0,\dots,\Delta_n$ are $n$-dimensional polytopes. Let $P \in I^{\sat}_{\delta}$ be any homogeneous polynomial  whose class in $(I^{\sat}/I)_{\delta}$ is nonzero. Then, for all $i = 0,\dots,n$, $P$ must have degree $\geq 1$ with respect to the coefficients of $F_i$.
\end{lemma}
\begin{proof}
    For simplicity, suppose that $P$ does not depend on the coefficients of $F_0$. For any maximal cone $\tau \in \Sigma(n)$, consider the monomial $x^{\mathbf{F}m_{0,\tau} + a_0}$ for some $m_{0,\tau} \in \mathcal{A}_i$, which is invertible in $C_{\Tilde{x}^\tau}$ by Lemma \ref{lemmaattheend}. Let  $c_{0,\tau}$ be the coefficient of $x^{\mathbf{F}m_{0,\tau} + a_0}$ in $F_0$ and consider $P$ as an element of $C_{\Tilde{x}^\tau}$. As $P \in I^{\sat}$, there must be $L \in \mathbb{Z}_{>0}$ such that:
    $$ (\Tilde{x}^{\tau})^L P = G_0F_0 + \dots + G_nF_n \in I.$$
    However, as $P$ does not involve $c_{0,\tau}$, we can change this coefficient in $C_{\Tilde{x}^\tau}$ by $\frac{ c_{0,\tau}x^{\mathbf{F}m + a_0} - F_0}{x^{\mathbf{F}m + a_0}}$ and without changing $P$. Therefore, $(\Tilde{x}^{\tau})^L P$ belongs to the ideal generated by $F_1,\dots,F_n$ in $C_{\Tilde{x}^\tau}$. Up to multiplying by some power $L' \geq L$, we must have that $(\Tilde{x}^{\tau})^{L'} P$ belongs to the ideal generated by $F_1,\dots,F_n$ in $C$. As the above conslusion holds for every $\tau \in \Sigma(n)$, we deduce that $P \in (F_1,\dots,F_n)^{\sat}_{\delta}$. Now, using that the polytopes $\Delta_0,\dots,\Delta_n$ are $n$-dimensional, Lemma \ref{lemmastrictusbset} implies that $P \in (F_1,\dots,F_n)_{\delta}$, contradicting that the class of $P$ modulo $I$ is nonzero.   
\end{proof}

\begin{proposition} 
\label{sylvesterform}
\cc{If the polytopes $\Delta_0,\dots,\Delta_n$ are $n$-dimensional,} the element $\Sylv_0$ belongs to $(I^{\sat})_\delta$. Moreover, 
$\sylv_0$ is independent of the choices of decompositions \eqref{eq:decompsylv}. In addition, if $\delta \notin \Gamma_0 \cup \Gamma_1$, then $\sylv_0$ is a generator of $(I^{\sat}/I)_{\delta}$ which is a free $A$-module of rank 1.
\end{proposition}
\begin{proof}
Note that if $\tau \in \Sigma(n)$, then either $\tau \neq \sigma$, in which case there is $k \in \{1,\dots,n\}$ such that $x_k$ divides $\Tilde{x}^{\tau}$ or $\tau = \sigma$, in which case $\Tilde{x}^{\tau} = z_1\cdots z_r$. Using the invariance of the determinant under column operations and using the decomposition in \eqref{eq:decompsylv}, we get
\begin{equation}\label{determinantsat}x_k\Sylv_{0} = \det\begin{blockarray}{ccc}
\begin{block}{(ccc)}
\cdots &  x_k F_{0,k} & \cdots\\
\cdots &  \vdots & \cdots\\
\cdots &  x_k F_{n,k} & \cdots\\
\end{block}
\end{blockarray} = \det\begin{blockarray}{ccc}
\begin{block}{(ccc)}
\cdots &  F_0 & \cdots\\
\cdots &  \vdots & \cdots\\
\cdots &  F_n & \cdots\\
\end{block}
\end{blockarray} \in I, \, k = 1,\dots,n.\end{equation}
The same holds for the monomial $z_1\cdots z_r$. Therefore, we deduce that $\Sylv_0 \in I^{\sat} = (I:\mathfrak{b}^{\infty})$. In order to prove that $\sylv_0$ has degree $\delta$, we find the degree of each entry $(i,j)$ of the matrix defined by the $F_{i,j}$'s. In \eqref{eq:decompsylv}, we divided the monomials of degree $\alpha_i$ by a monomial of degree 
$$\begin{cases}\pi(e_j) & \text{if the monomial is } x_k \text{ for } k = 1,\dots,n, \\ \pi(\sum_{j = 1}^{r} e_{n+j}) & \text{if the monomial is } z_1\cdots z_r, \end{cases}$$ where  $\{e_j\}_{j = 1}^{n+r}$ is the canonical basis of $\mathbb{Z}^{\Sigma(1)}$. On the other hand, the anticanonical class $K_X$ coincides with the degree of the monomial $x_1\cdots x_n z_1 \cdots z_r$ (see \cite[Theorem 8.2.3]{coxlittleschneck}), which is equal to $\pi(\sum_{j = 1}^{n+r}e_j)$. Therefore, the degree of each of the summands constituting the determinant is equal to:
\begin{equation}
\label{determinant}
    \sum_{i = 0}^n \big( \alpha_i - \pi(e_{\tau(i)}) \big) = \left(\sum_{i = 0}^n \alpha_i \right)- K_X = \delta,
\end{equation}
where $e_0 = \sum_{k=n+1}^{n+r}e_k$ and $\tau$ is any permutation of $\{0,\dots,n\}$. 

The fact that $\sylv_0$ is nonzero and the independence from the choice of the decompositions in \eqref{eq:decompsylv} are consequences of the global transformation law; see \cite[Remark 2.12 iii), iv)]{cattani1995residues}.

\cc{If $\delta \notin \Gamma_0 \cup \Gamma_1$,  $(I^{\sat}/I)_{\delta}$ is a  free $A$-module of rank one. By Lemma \ref{allthevariables}, any generator $g$ of $(I^{\sat}/I)_{\delta}$ must have degree greater or equal than $1$ with respect to the coefficients of $F_i$ for $i = 0,\dots,n$. On the other hand, the construction of $\sylv_0$ indicates that for all $i = 0,\dots,n$, the degree of $\sylv_0$ with respect to the coefficients of $F_i$ is smaller or equal to 1. Thus, if we write $\sylv_0 = cg$ for some $c \in A$, the degree of $c$ with respect to $A$ must be zero, implying that $c \in \mathbf{k}$. This implies that $\sylv_0$ is also a generator of $(I^{\sat}/I)_{\delta}$ as an $A$-module.}
\end{proof}

\cc{In order to use \cite[Remark 4.12 iv)]{cattani1995residues}, we need to be able to specialize to values in the field of complex numbers $\mathbb{C}$. Therefore, from now on, we assume that the field $\mathbf{k}$ is a \textit{subfield of the complex numbers.}} \lb{Assuming $\delta \notin \Gamma_0 \cup \Gamma_1$, Theorem \ref{batyrevborisov} implies that the Sylvester form $\sylv_0$ corresponds} to the unique lattice point in the interior of the polytope $\Delta_{\Sigma}$ associated to the anticanonical divisor $K_X$, i.e.
\[ (I^{\sat}/I)_{\delta } \simeq H^{n+1}_{\mathfrak{b}}(C(-\sum \alpha_i))_{\delta} \simeq H^{n}(X_{\Sigma}, - K_X) \simeq \oplus_{m \in \Relint( \Delta_{\Sigma})}A\chi^{-m}.\]    

So far, we proved that the toric Jacobian $\sylv_0$ yields an $A$-basis of $(I^{\sat}/I)_{\delta}\simeq A$. The next step is to construct an $A$-basis of $(I^{\sat}/I)_{\delta - \nu}$ when it is a free  $A$-module.

\begin{definition}
Let $X_{\Sigma}$ be a projective toric variety which is $\sigma$-positive for some $\sigma \in \Sigma(n)$.  Assume that the polytopes $\Delta_0,\dots,\Delta_n$ are $n$-dimensional. 
Let $\nu \in \Cl(X_{\Sigma})$ be a nef Cartier class and $\Delta_{\nu}$ be the corresponding polytope written as in \eqref{eq:polytopesnu} and satisfying $0 \leq \nu_{n+j} < \min_{i = 0,\dots,n} a_{i,n+j}$ for $j = 1,\dots,r$.  According to Theorem \ref{theoremdecompositionrestated}, 
for any  $x^{\mu} \in R_{\nu}$ and for any $i\in\{0,\ldots,n\}$ the polynomial $F_i$ can be decomposed as
\begin{equation}
\label{eq:Fidecomp2}
F_i = z_1^{\mu_{n+1}+1}\cdots z_r^{\mu_{n+r}+1}F^{\mu}_{i,0} + x_{1}^{\mu_1+1}F^{\mu}_{i,1} + \dots + x_n^{\mu_n + 1}F^{\mu}_{i,n}.
\end{equation}
We define the \textit{toric Sylvester form} $\Sylv_{\mu}$ as the determinant 
\[ \Sylv_{\mu} = \det(F^{\mu}_{i,j})_{0\leq i,j\leq n}.\]
The  class of $\Sylv_{\mu}$ modulo $I$ is denoted by $\sylv_{\mu}$. Observe that, as with $\Sylv_0$, the Sylvester forms are linear in the coefficients of $F_i$ for $i = 0,\dots,n$.
\end{definition}


\cc{
If we are given two different monomials $x^{\mu}, x^{\mu'} \in C_{\nu}$, there must be some $k \in \{1,\dots,n\}$ such that $\mu_{k} \neq \mu'_k$. Otherwise, using \eqref{eq:murelation}, we can derive that $x^{\mu} = x^{\mu'}$. With this, we can introduce the following  \textit{lexicographical monomial order}.
\begin{definition}
\label{definitionorder}
Given two monomials $x^{\mu}$ and $x^{\mu'}$ of degree $\nu$, we say $\mu < \mu'$ if $k_0 = \min\{k \in \{1,\dots,n\} \quad \mu_{k} \neq \mu_k\}$ satisfies $\mu_{k_0} < \mu'_{k_0}$.
\end{definition}}

\begin{theorem}\label{sylvesterformduality} 
Let $X_{\Sigma}$ be a projective toric variety which is $\sigma$-positive for some $\sigma \in \Sigma(n)$ and \cc{that the polytopes $\Delta_0,\dots,\Delta_n$ are $n$-dimensional}. Let $\nu \in \Cl(X_{\Sigma})$ be a class satisfying the hypotheses of Theorem \ref{theoremdecompositionrestated}. Then, for every $x^\mu \in R_\nu$, $\Sylv_\mu$ belongs to $(I^{\sat})_{\delta - \nu}$ and its class $\sylv_{\mu}$ is a nonzero element in $(I^{\sat}/I)_{\delta - \nu}$ . Moreover, for $x^{\mu},x^{\mu'} \in R_{\nu}$, we have
\begin{equation}\label{duality} x^{\mu'}\sylv_{\mu} = \begin{cases} \sylv_0 & \mu = \mu'  \\ 0 & \mu < \mu'\end{cases}\end{equation}
As a consequence, the Sylvester forms $\{\sylv_{\mu}\}_{x^{\mu} \in R_{\nu}}$ are linearly independent in $(I^{\sat}/I)_{\delta-\nu}$.
\end{theorem}

\begin{proof}
The fact that $\Sylv_\mu$ is of degree $\delta - \nu$ follows by analyzing the degree of each summand in $\det(F^{\mu}_{i,j})$ as in \eqref{determinant}. Moreover, we can use the same argument as in \eqref{determinantsat}, to see that $x_k^{\mu_k+1}\Sylv_{\mu} \in I$ for all $k = 1,\dots,n$ and  $z_1^{\mu_{n+1}+1}\cdots z_r^{\mu_{n+r}+1}\Sylv_{\mu} \in I$. This proves that $\Sylv_{\mu} \in I^{\sat}_{\delta-\nu}$. Consider two distinct monomials $x^{\mu}, x^{\mu'} \in R_{\nu}$ such that $\mu < \mu'$, then there is $k_0 \in \{1,\dots,n\}$ such that:
\[ x^{\mu'}\Sylv_{\mu} = \frac{x^{\mu'}}{x_{k_0}^{\mu_{k_0} + 1}}x_{k_0}^{\mu_{k_0} + 1}\Sylv_{\mu} \in I\]
and hence $x^{\mu'}\sylv_{\mu} = 0 \in (I^{\sat}/I)_{\delta - \nu}$. On the other hand, we have:

\[ x^{\mu}\Sylv_{\mu} = x_1^{\mu_1}\cdots x_n^{\mu_n} z_1^{\mu_{n+1}}\cdots z_r^{\mu_{n + r}}\det(F^{\mu}_{i,j}) = \det(x_j^{\mu_j}F^{\mu}_{i,j})\]
but at the same time, the decomposition
\[ F_i = z_1\cdots z_r z_1^{\mu_{n+1}}\cdots z_r^{\mu_{n+r}}F^{\mu}_{i,0} + x_1x_1^{\mu_1}F^{\mu}_{i,1} + \dots + x_nx_n^{\mu_2}F^{\mu}_{i,n} \]
gives the Sylvester form $\sylv_0$, implying that \cc{$x^{\mu}\sylv_{\lb{\mu}} = \sylv_0$} and that $\Sylv_{\mu} \notin I$. \cc{From these two facts, we can derive that the Sylvester forms are nonzero in $(I^{\sat}/I)_{\delta-\nu}$ and linearly independent. Namely, if we have a relation $\sum_{x^{\mu} \in R_{\nu}}\lambda_{\mu} \sylv_{\mu} = 0$ for some $\lambda_{\mu} \in A$, then multiplying by the monomials $x^{\mu} \in R_{\nu}$ in decreasing order with respect to $<$, we derive that $\lambda_{\mu} = 0$ for all $x^{\mu} \in R_{\nu}$.} 
\end{proof}

We notice that the relation between Sylvester forms and monomials stated in Theorem \ref{sylvesterformduality} can also be deduced from the global transformation law in \cite{cattani1995residues}. As the decomposition we provided in Theorem \ref{theoremdecompositionrestated} differs from the one provided in \cite[Section 4]{buse2021multigraded} for the multihomogeneous case, we can see that, in general, the Sylvester forms and the monomials of degree $\nu$ do not form a pairing. In particular, we can see that there is a matrix $\mathcal{D} = (\mathcal{D}_{\mu,\mu'})_{x^{\mu},x^{\mu'} \in R_{\nu}}$ whose entries are polynomials in $A$ ordered with respect to $<$ and satisfy that:
\begin{equation}
\label{equationnotpairing}
x^{\mu'}\sylv_{\mu} = \mathcal{D}_{\mu,\mu'}\sylv_0.
\end{equation}
Note that $\mathcal{D}_{\mu,\mu'}$ can be computed using the global transformation law and noting that:
\begin{equation}\label{equationresidue}\mathcal{D}_{\mu,\mu'} = \Residue_{(F_0,\dots,F_n)}(x^{\mu'}\sylv_{\mu}) = \Residue_{(x_1^{\mu_1+1},\dots,x_n^{\mu_n+1},z_1^{\mu_{n+1}+1}\cdots z_r^{\mu_{n+r} + 1})}(x^{\mu'})\end{equation}
This last residue is zero, if and only if, $x^{\mu'}$ belongs to the ideal $(x_1^{\mu_1+1},\dots,x_n^{\mu_n+1},z_1^{\mu_{n+1}+1}\cdots z_r^{\mu_{n+r} + 1})$. Otherwise, as the residue does not depend on $A$, it must be a nonzero element in $\mathbf{k}$, which is also independent of the decomposition \eqref{eq:decomposition} giving rise to $\sylv_{\mu}$. Theorem \ref{sylvesterformduality} implies that the matrix $\mathcal{D}$ is lower triangular with ones in the diagonal. Therefore, $\mathcal{D}$ is invertible and its inverse has entries in $\mathbf{k}$. If $X_{\Sigma} = \mathbb{P}^n$, the decomposition in \eqref{eq:Fidecomp2} coincides with the one given in \cite{joanolouformesdinertie} and we can see that $\mathcal{D}$ is the identity matrix. We refer to Section \ref{sec:residues} for more details on the toric residue and an example where we explicitly show that the matrix $\mathcal{D}$ is not the identity matrix; see Example \ref{example:residues}.
\medskip

In the next theorem, we prove that the Sylvester forms yield an $A$-basis of $(I^{\sat}/I)_{\delta}$, when it is a free $A$-module. This result is the key to the applications we discuss in the following sections.

\begin{theorem}
\label{sylvesterformsbasis}
    Under the assumptions of Theorem \ref{sylvesterformduality} and if $\delta-\nu \notin \Gamma_0\cup\Gamma_1$ (see Notation \ref{notation}), $\{ \sylv_{\mu} \}_{x^\mu \in C_{\nu}}$ is an $A$-basis of $(I^{\sat}/I)_{\delta - \nu}$. Moreover, the classes $\sylv_{\mu}$ do not depend on the choice of the decompositions in \eqref{eq:Fidecomp2}.
\end{theorem}
\begin{proof}
    In Theorem \ref{sylvesterformduality}, we proved that the set of forms $\{\sylv_{\mu}\}_{x^{\mu} \in R_{\nu}}$ is linearly independent. Moreover, as in \cite[Theorem 4.9]{buse2021multigraded}, consider the canonical basis of $\Hom(C_{\nu}, A)$ which is dual to the monomial basis of $C_{\nu}$. Namely, to each monomial  $x^{\mu} \in C_{\nu}$, we associate the map:
    $$X^{\mu}: C_{\nu} \xrightarrow[]{} A $$ which sends $x^{\mu}$ to one and every other monomial to $0$. Moreover,  consider the $A$-linear isomorphisms $$\phi: A \xrightarrow[]{} (I^{\sat}/I)_{\delta} \quad c \xrightarrow[]{} c\sylv_0$$ and 
    $$\mathcal{D}_{\nu}: C_{\nu} \xrightarrow[]{} C_{\nu} \quad x^{\overline{\mu}} \xrightarrow[]{} \sum_{x^{\mu'} \in R_{\nu}} \mathcal{D}_{\mu',\overline{\mu}}x^{\mu'}$$
    where $\phi$ is an isomorphism by Proposition \ref{sylvesterform} and $\mathcal{D}_{\nu}$ is an isomorphism because the matrix $\mathcal{D}$ is invertible. Therefore, by \eqref{equationnotpairing}, the composition $\phi \circ X^{\mu} \circ \mathcal{D}_{\nu}$ corresponds to multiplying the monomials in $C_{\nu}$ by $\sylv_{\mu}$ and realizes the isomorphism $(I^{\sat}/I)_{\delta-\nu} \simeq \Hom(C_{\nu},A)$. This proves that the Sylvester forms $\{\sylv_{\mu}\}_{x^{\mu} \in R_{\nu}}$ yield an $A$-basis of $(I^{\sat}/I)_{\delta-\nu}$. Moreover, this also implies that the classes $\sylv_{\mu}$ are independent of the decompositions \eqref{eq:Fidecomp2} since the maps  $\phi \circ X^{\mu} \circ \mathcal{D}_{\nu}$ are themselves independent of these decompositions.  
\end{proof} 

\section{Application to toric elimination matrices}
\label{sec::hybridmat}

An important motivation for studying the structure of the saturation of an ideal generated by generic sparse polynomials is for applications in elimination theory, in particular for solving sparse polynomial systems. In this section, we introduce a family of matrices whose construction involves toric Sylvester forms. It yields new compact elimination matrices that can be used for solving $0$-dimensional sparse polynomial systems via linear algebra methods. We refer the reader to \cite{emirismourrain,bender2021toric, telenthesis} for a thorough exposition of such methods that we will not discuss in this paper. 

In what follows, \lb{$X_{\Sigma}$ will denote a projective toric variety which is assumed to be smooth and $\sigma$-positive for some maximal cone $\sigma \in \Sigma(n)$, and we will consider a generic sparse polynomial system defined by homogeneous polynomials $F_0, \dots, F_n$ as defined in \eqref{eq:Fi}. We require $X_{\Sigma}$ to be smooth} because we will use the Grothendieck-Serre formula \eqref{grserre}. This setting covers many cases that are of interest for applications. We notice that the smoothness assumption is not very restrictive as $X_{\Sigma}$ can be replaced by one of its  desingularization \cc{varieties} (see e.g.~\cite[Chapters 10,11]{coxlittleschneck}),  \lb{but the preservation of the \cc{$\sigma$-positive} property under desingularization is not obvious}.

\begin{notation}\label{notSec5} The elimination matrices we will consider are universal with respect to the coefficients of the $F_i$'s, so we introduce the following notation to study rigorously their properties under specialization of these coefficients. Recall that $I$ denotes the ideal in $C$ generated by $F_0,\ldots,F_n$.

Any specialization (i.e.~ring morphism) $\theta: A \xrightarrow[]{} \mathbf{k}$ induces a surjective map $C \xrightarrow[]{} R$ where $R = \mathbf{k}[x_{\rho} \, : \, \rho \in \Sigma(1)]$ (this map leaves invariant the variables $x_\rho$). For all $i = 0,\ldots,n$, we define $f_i = \theta(F_i) \in R$, we denote by  $I(f)$ the homogeneous ideal $( f_0,\ldots,f_n )$ of $R$ and set $B(f) = R/I(f)$. Moreover, we also set $B^{\sat} = C/I^{\sat}$, $B(f)^{\sat} = R/I(f)^{\sat}$ and $B^{\sat}(f) = C/ I^{\sat}(f)$ (observe that $I(f)^{\sat}$ and $I^{\sat}(f)$ are in general not the same ideals). Finally, for any matrix $\mathbb{M}$ with coefficients in $A$, we denote by $\mathbb{M}(f)$ its specialization by $\theta: A \xrightarrow[]{} \mathbf{k}$. We will refer as $V(I(f))$ to the zero set of the polynomial system defined by $I(f)$ over $(X_{\Sigma})_{\overline{\mathbf{k}}}$ where $\overline{\mathbf{k}}$ denotes an algebraic closure of $\mathbf{k}$ \footnote{In Section \ref{sec:sylvforms}, we assumed that $\mathbf{k}$ is a subfield of $\mathbb{C}$. Thus, we can consider that $V(I(f))$ is the zeros of $I(f)$ over $\mathbb{C}$.}.

\lb{In what follows, } we will consider $\Pic(X_{\Sigma})$ instead of $\Cl(X_{\Sigma})$ as all Weil divisors are Cartier in a smooth variety (see \cite[Proposition 4.2.6]{coxlittleschneck}).
\end{notation}

\subsection{Hybrid elimination matrices}\label{subsec:hybridmat}

We begin by describing precisely what we mean by an elimination matrix $\mathbb{M}$ associated to the polynomials $F_0,\ldots,F_n$. It is a matrix 
whose columns are filled with coefficients of some homogeneous forms that are of the same degree and that all belong to the saturated ideal $I^{\sat} \subset C$. Thus, its entries are polynomials in $A$. Moreover, it is required that for any specialization map $\theta: A \xrightarrow[]{} \mathbf{k}$ 
the following two properties hold:
\begin{itemize}
    \item[i)] The corank of $\mathbb{M}(f)$ is equal to zero, if and only if, $f_0 = \dots = f_n = 0$ has no solution in $X_{\Sigma}$. 
    \item[ii)] If the number of solutions of $f_0 = \dots = f_n = 0$ (over $\overline{\mathbf{k}}$) is finite in $X_{\Sigma}$ and equals $\kappa$, then the corank of $\mathbb{M}(f)$ is  $\kappa$.
\end{itemize}

\lb{We note that the first property yields a certificate of existence of a common root of the $f_i$'s, which is related to sparse resultants, a topic we will address in the next section}. The second property is mainly required for solving 0-dimensional polynomial systems by means of linear algebra techniques based on eigen-computations. In this approach, the common roots of the $f_i$'s are extracted from the cokernel of $\mathbb{M}(f)$ \lb{(see e.g.~\cite{yetbendertelen})}. 

A very classical family of elimination matrices is obtained by filling columns with all the multiples of the $F_i$'s of a certain degree. These matrices are usually called Macaulay-type matrices and are widely used for solving 0-dimensional polynomial systems (see for instance \cite{bender2021toric}). To be more precise, these matrices, 
 that we will denote by $\mathbb{M}_{\alpha}$, are presentation matrices of the $A$-module $B_\alpha$, i.e.~are matrices of the maps
\begin{eqnarray}\label{eq:Malphamap}
\label{nonhybrid}
    \big(\oplus_{i = 0}^nC(-\alpha_i)\big)_{\alpha} & \xrightarrow[]{} &  C_\alpha \\ 
	\nonumber (G_0,\dots,G_n) & \mapsto & \sum_{i = 0}^nG_iF_i.	
\end{eqnarray}
Of course, some conditions on $\alpha \in \Pic(X_{\Sigma})$ are required in order to guarantee that $\mathbb{M}_\alpha$ is an elimination matrix; we refer to \cite{emirismourrain} and to \cite[Chapter 5]{telenthesis} for more details. Applying results we proved in the previous sections, we are going to extend the family of Macaulay-type matrices by using toric Sylvester forms. We recall that Sylvester forms belong to $I^{\sat}$ by Theorem \ref{sylvesterformduality}.

\begin{definition}\label{def:H_alpha}
Let $\alpha$ be such that $\left(I^{\sat}/I\right)_\alpha$ is a free $A$-module \lb{generated by Sylvester forms, so that $\left(I^{\sat}/I\right)_\alpha \simeq \oplus_{x^\mu \in C_{\delta - \alpha}} A$ (see Corollary \ref{cor:freemodule} and Theorem \ref{sylvesterformsbasis}), and consider the map}
\begin{eqnarray}\label{hybrid}
    \big(\oplus_{i = 0}^nC(-\alpha_i)\big)_{\alpha} \oplus \left(\bigoplus_{x^\mu \in C_{\delta - \alpha}} A \right) & \rightarrow &  C_{\alpha}\\ \nonumber (G_0,\ldots,G_n)\oplus (\ldots, \ccc{\ell}_{\mu},\ldots ) &  \mapsto & \sum_{i = 0}^n G_iF_i + \sum_{x^\mu \in C_{\delta - \alpha}}\ccc{\ell}_{\mu}\Sylv_{\mu}.	
\end{eqnarray}
Its matrix \lb{(in canonical bases)} is called a \textit{hybrid elimination matrix} and denoted by $\mathbb{H}_{\alpha}$.
\end{definition}
 The matrices  $\mathbb{H}_{\alpha}$ are called \emph{hybrid} because they are composed of two blocks, one from the classical Macaulay-type matrices and another one built from toric Sylvester forms; \cc{see Example \ref{example:h2}}. In particular,   $\mathbb{M}_\alpha=\mathbb{H}_\alpha$ if $(I^{\sat}/I)_{\alpha} = 0$, so that the family of matrices $\mathbb{H}_\alpha$  extends  the family of Macaulay-type matrices $\mathbb{M}_\alpha$. \lb{Thus, from now on} we will use the notation $\mathbb{H}_\alpha$ instead of $\mathbb{M}_\alpha$. Our next step is to prove that these matrices are elimination matrices.
 
\subsection{Main properties}

 In this section, we first prove that the matrices $\mathbb{H}_\alpha$ introduced in Definition \ref{def:H_alpha} are elimination matrices. Then, we give an illustrative example and also provide another criterion to construct the matrices $\mathbb{H}_\alpha$ without relying on the computation of the supports $\Gamma_0$ and $\Gamma_1$ \lb{(see Notation \ref{notation})}. 
 
First, suppose given a specialization map (see Notation \ref{notSec5}) and a degree $\alpha$. From the results of Section \ref{sec:duality} and Section \ref{sec:sylvforms}, and also Definition \ref{def:H_alpha}, we deduce that the image of the matrix $\mathbb{H}_\alpha(f)$ is $I^{\sat}(f)_\alpha$, so that its corank is $\HF(B^{\sat}(f),\alpha)$. Therefore, a natural question is to compare this Hilbert function of $B^{\sat}(f)$ with the one of $B(f)^{\sat}$ in degrees for which hybrid matrices $\mathbb{H}_{\alpha}$ are defined (see Definition \ref{def:H_alpha}). We recall that we \lb{use} the notation of Section \ref{sec:duality} and we assume that the toric variety $X_\Sigma$ is smooth and $\sigma$-positive for a maximal cone $\sigma \in \Sigma(n)$.

\begin{lemma}\label{lem:HFBsat}
Let $\alpha \notin \Gamma_0 \cup \Gamma_1 \subset \Pic(X_{\Sigma})$ and suppose given specialized polynomials $f_0,\ldots,f_n$ defining a 0-dimensional subscheme in $X_\Sigma$, possibly empty, of $\kappa$ points, counted with multiplicity. Then, \[\HF(B(f)^{\sat},\alpha)=\HF(B^{\sat}(f),\alpha)=\kappa.\]
\end{lemma}
\begin{proof} This proof goes along the same lines as \cite[Lemma 2.7]{buse2021multigraded}. First, one observes that  $I(f)\subset I^{\sat}(f) \subset I(f)^{\sat}$ so that $B(f)^{\sat}$, $B^{\sat}(f)$ and $B(f)$ have the same Hilbert polynomial, which is the constant $\kappa$ by our assumption. 

Now, $H^i_{\mathfrak{b}}(B(f)^{\sat}) = 0$ for $i=0$ and for all $i>1$ since $V(I(f))$ is finite. Applying Grothendieck-Serre formula, it follows that $\HF(B(f)^{\sat},\alpha)=\kappa$ for all $\alpha$ such that $H^1_{\mathfrak{b}}(B(f)^{\sat})_\alpha=0$. Analyzing the two spectral sequences associated to the Čech-Koszul complex of $f_0,\ldots,f_n$, we get that the above vanishing holds for all $\alpha \notin \Gamma_0 \cup \Gamma_1$.

Similarly, Grothendieck-Serre formula and the finiteness of $V(I(f))$ imply that $\HF(B^{\sat}(f),\alpha)=\kappa$ for all $\alpha$ such that $H^0_{\mathfrak{b}}(B(f)^{\sat})_\alpha=H^1_{\mathfrak{b}}(B(f)^{\sat})_\alpha=0$. By \cite[Proposition 6.3]{chardinpowers}, the vanishing of these modules can be derived from the similar vanishing conditions  $H^0_{\mathfrak{b}}(B^{\sat})_\alpha=H^1_{\mathfrak{b}}(B^{\sat})_\alpha=0$. These latter conditions hold for all $\alpha \notin \Gamma_0 \cup \Gamma_1$, which concludes the proof.
\end{proof}

\begin{remark} As a consequence of the above lemma, the canonical map from $I^{\sat}_\alpha$ to $I(f)^{\sat}_\alpha$, which is induced by a specialization $\theta$, is surjective, i.e.~generators of $I(f)^{\sat}_\alpha$ can be computed by means of universal formulas. 
\end{remark}

\begin{theorem}
\label{prop:Helimmat} 
Assume that the toric variety $X_\Sigma$ is smooth and $\sigma$-positive for a maximal cone $\sigma\in \Sigma(n)$.
Then, for any $\alpha \notin \Gamma_0 \cup \Gamma_1 \subset \Pic(X_{\Sigma})$ \cc{satisfying that $\left(I^{\sat}/I\right)_\alpha\simeq \oplus_{x^\mu \in C_{\delta-\alpha}} A$}, the matrix $\mathbb{H}_\alpha$ is an elimination matrix, i.e.~it satisfies:
\begin{itemize}
    \item[i)] $\mathrm{corank}(\mathbb{H}_\alpha(f))=0$ if and only if $V(I(f))$ is empty in $X_\Sigma$,
    \item[ii)] If $V(I(f))$ is a finite subscheme of degree $\kappa$ in $X_\Sigma$, then $\mathrm{corank}(\mathbb{H}_\alpha(f))=\kappa$.
\end{itemize}
\end{theorem} 
\begin{proof} \lb{We first prove i)}. If $V(I(f))$ is empty, equivalently $B(f)^{\sat}=0$ \footnote{This equivalence follows by the Grothendieck-Serre formula which requires the smoothness of $X_\Sigma$.}, then $\HF(B^{\sat}(f),\alpha)=0$ by Lemma \ref{lem:HFBsat}. If $V(I(f))\neq \emptyset$, then the $f_i$'s have a common solution, say the point $p \in X_\Sigma$ (over $\mathbf{k}$) with defining ideal $I_p$ (radical and maximal in $R$). Therefore, since $I^{\sat}(f) \subset I(f)^{\sat}\subset I_p$
and $\HF(R/I_p,\beta) = 1$ for all $\beta \in \Pic(X_{\Sigma})$ by the maximality of $I_p$, we deduce that $\HF(R/I^{\sat}(f),\alpha) \neq 0$ for any $\alpha$.  The proof of ii) follows from Lemma \ref{lem:HFBsat}. 
\end{proof}

\begin{figure}[H]
    \centering
\begin{tikzpicture}

\draw[blue] (0,0) -- (2,0);
\draw[blue] (1,1) -- (2,0);
\draw[blue] (1,1) -- (0,1);
\draw[blue] (0,0) -- (0,1);

\filldraw[red] (0,0) circle (2pt) node[anchor=north] {\tiny (0,0)};
\filldraw[red] (1,0) circle (2pt) node[anchor=north] {\tiny (1,0)};
\filldraw[red] (2,0) circle (2pt) node[anchor=north] {\tiny (2,0)};
\filldraw[red] (0,1) circle (2pt) node[anchor=south] {\tiny (0,1)};
\filldraw[red] (1,1) circle (2pt) node[anchor=south] {\tiny (1,1)};

\end{tikzpicture}
    \caption{The polytopes $\Delta_i$ corresponding to the generic sparse homogeneous polynomials in Example \ref{example:h2} with the lattice points marked in red.}
    \label{fig:polytope}
\end{figure}

\begin{example}
\label{example:h2}
Let $M = \mathbb{Z}^2$ and $X_{\Sigma}$ be the Hirzebruch surface $\mathcal{H}_1$ described in Example \ref{example:h1}. Consider the following polytope presentations:
\[\Delta_i = \{m \in \mathbb{R}^2 \, : \, \langle m,(1,0) \rangle \geq 0, \: \langle m, (0,1) \rangle \geq 0, \: \langle m, (-1,-1)\rangle \geq -2, \: \langle m, (0,-1) \rangle \geq -1 \}, \, i=0,1,2.\]
$\mathcal{H}_1$ has the $\sigma$-positive property for $\sigma = \langle (1,0), (0,1)\rangle$. The class in $\Pic(\mathcal{H}_1) = \mathbb{Z}^2$ corresponding to these polytopes is $\alpha_i = (2,1)$, \lb{$i=0,1,2$,} and we write the corresponding generic homogeneous sparse polynomials as:
\begin{equation}\label{eq:F0F1F2}
F_0 = a_0z_1^2z_2 + a_1x_1z_1z_2 + a_2x_1^2z_2 + a_3x_2z_1 + a_4x_1x_2 \quad \text{resp. } F_1,F_2 \text { with coefficients } b_j,c_j, j = 0,\dots,4.
\end{equation}
\begin{figure}[H]

  \begin{floatrow}
\floatbox[{\capbeside\thisfloatsetup{capbesideposition={right,center},capbesidewidth=10cm}}]{figure}{     
\begin{tikzpicture}[scale=0.75]

    
    \draw[pattern=north east lines, pattern color=blue] (0,1/2) -- (3/2,2) -- (-2,2) -- (-2,1/2) -- (0,1/2);
    
    \draw[pattern=north east lines, pattern color=blue] (1/2,-1/2) -- (-1,-2) -- (3,-2) -- (3,-1/2) -- (1/2,-1/2);
    
    \draw[pattern=north west lines, pattern color=red] (3/2,0) -- (-1/2,-2) -- (3,-2) -- (3,0) -- (3/2,0);
    
    \draw[pattern=north west lines, pattern color=red] (1,1) -- (2,2) -- (-2,2) -- (-2,1) -- (1,1);
    
    \draw[pattern=north east lines, pattern color=green] (2,3/2) -- (5/2,2) -- (-2,2) -- (-2,3/2) -- (2,3/2);
    
    \draw[pattern=north east lines, pattern color=green] (5/2,1/2) -- (0,-2) -- (3,-2) -- (3,1/2) -- (5/2,1/2);

    \draw[pattern=north east lines, pattern color=brown] (2,1) -- (2,2) -- (3,2) -- (3,1) -- (2,1);

        \filldraw[black] (3/2,1/2) circle (2pt) node[anchor= north] {\scriptsize (3,1)};
        
        \filldraw[black] (3/2,1) circle (2pt) node[anchor= east] {\scriptsize (3,2)};

        \filldraw[black] (2,1) circle (2pt) node[anchor= south] {\scriptsize(4,2)};

        \filldraw[black] (1,1/2) circle (2pt) node[anchor= east] {\scriptsize (2,1)};

        \filldraw[orange] (1,1/2) circle (1pt) node[anchor= east] {};

        \filldraw[orange] (3/2,1/2) circle (1pt) node[anchor= north] {};

    \end{tikzpicture} }{
    \caption{\footnotesize This is the picture of the regions $\Gamma_0,\Gamma_1,\Gamma_{\Res},\Gamma \subset \Pic(X_{\Sigma}) = \mathbb{Z}^2$ (the latter being defined in Section \ref{sec:resultants}, \eqref{gammak2}). The blue region corresponds to $\Gamma_0$, the red region corresponds to $\Gamma_1$, the green region corresponds to $\Gamma_{\Res}$ and the brown region corresponds to $\Gamma$. We marked in orange those $\alpha$ with $(I^{\sat}/I)_{\alpha} \neq 0$. We derived the local cohomology of $\mathcal{H}_1$ from \cite{altmann2018immaculate}; see also \cite{eisenbud2000cohomology,botbol2011implicit}.}  \label{figure1}} 

 \end{floatrow}
\end{figure}

\noindent Figure \ref{figure1} describes the supports $\Gamma_0,\Gamma_1,\Gamma_{\Res}$. We deduce that elimination matrices $\mathbb{H}_{\alpha}$ are obtained for $\alpha \in \{ (4,2), (3,2), (3,1), (2,1)\}$. In the cases $\alpha=(4,2)$ and $\alpha=(3,2)$, we get two Macaulay-type matrices. The two other cases give the following matrices:

\begin{itemize}

\item Case $\alpha = (3,1)$. This matrix corresponds to $\alpha = \delta$ and in this case, we are introducing a Sylvester form. This form is $\Sylv_0$ and can be computed, as before, by a determinant that we write as:

\[ \det\begin{blockarray}{ccc}
\begin{block}{(ccc)}
 a_1z_1z_2 + a_2x_1z_2 + a_4x_2 &  a_3z_1 & a_0z_1\\
 b_1z_1z_2 + b_2x_1z_2 + b_4x_2 &  b_3z_1 & b_0z_1\\
 c_1z_1z_2 + c_2x_1z_2 + c_4x_2 &  c_3z_1 & c_0z_1\\
\end{block}
\end{blockarray} = [130]z_1^3z_2 + [230]x_1z_1^2z_2 + [430]x_2z_1^2, \] 
where $[ijk] = \scalemath{0.8}{
\det\begin{blockarray}{ccc}
\begin{block}{(ccc)}
 a_i & a_j & a_k \\
 b_i & b_j & b_k \\
 c_i & c_j & c_k  \\
\end{block}
\end{blockarray}}.$ Therefore, the elimination matrix $\mathbb{H}_{\cc{(3,1)}}$ is of the form:
$$
\mathbb{H}_{(3,1)}=
\setcounter{MaxMatrixCols}{9}
\begin{pmatrix} \Transpose{
a_0 & a_1 & a_2 & 0 & a_3 & a_4 & 0 \\
0 & a_0 & a_1 & a_2 & 0 & a_3 & a_4 \\
b_0 & b_1 & b_2 & 0 & b_3 & b_4 & 0 \\
 0 & b_0 & b_1 & b_2 & 0 & b_3 & b_4 \\
c_0 & c_1 & c_2 &  0 & c_3 & c_4 & 0 \\
0 & c_0 & c_1 & c_2 & 0 & c_3 & c_4 \\
[130] & [230] & 0 & 0 & [430] & 0 & 0
}
\end{pmatrix}.
$$
This type of matrices for $\alpha = \delta$ were already known from \cite{cattani1997residuesresultants} as the $\Delta_i$'s are all equal and ample in $\mathcal{H}_1$. 

\item Case $\alpha = (2,1)$. We obtain the following matrix $\mathbb{H}_{\cc{(2,1)}}$ which is built from two different Sylvester forms:

\[
\mathbb{H}_{(2,1)}=
\begin{pmatrix} \Transpose{
a_0 & a_1 & a_2 & a_3 & a_4 \\
b_0 & b_1 & b_2 & b_3 & b_4 \\
c_0 & c_1 & c_2 & c_3 & c_4 \\
\left[013\right] & \left[023\right] + [014] & [024] & 0 & 0 \\
\left[023\right] & [024] + [123] & [124] & 0 & 0
}
\end{pmatrix} 
\]
that correspond to the monomial basis $\{z_1,x_1\}$ of $C_{\nu}$ for $\nu = (1,0)$. As far as we know, \lb{these matrices did not appear in the existing literature}.  
\end{itemize}

\end{example}

\begin{example}
\label{example:h3}
Consider again Example \ref{example:h2} with the same $F_0,F_1$ as in \eqref{eq:F0F1F2} but suppose now that $\alpha_2 = (1,1)$ and thus the corresponding generic homogeneous sparse polynomial  $F_2$ is:
\begin{align}F_2 &= c_0z_1z_2 + c_1x_1z_2 + c_3x_2.
\end{align}
In this case, the Newton polytopes $\Delta_i$'s are not scaled copies of a fixed ample class and $\alpha_2$ is not even ample in $\mathcal{H}_1$. However, \cc{the polytopes $\Delta_i$ are $n$-dimensional. Therefore, Corollary \ref{corollarydecomposition} and Theorem \ref{prop:Helimmat} imply that $\mathbb{H}_{\delta}$ is an elimination matrix for $\delta = (2,1)$. T}he corresponding Sylvester form is \[ \det\begin{blockarray}{ccc}
\begin{block}{(ccc)}
 a_1z_1z_2 + a_2x_1z_2 + a_4x_2 &  a_3z_1 & a_0z_1\\
 b_1z_1z_2 + b_2x_1z_2 + b_4x_2 &  b_3z_1 & b_0z_1\\
 c_1z_2 &  c_3 & c_0\\
\end{block}
\end{blockarray} = [130]z_1^2z_2 + [230]x_1z_1z_2 + [430]x_2z_1,
\]
where 
$[ijk] := \scalemath{0.8}{
\det \begin{blockarray}{ccc}
\begin{block}{(ccc)}
 a_i & a_j & a_k \\
 b_i & b_j & b_k \\
 c_i & c_j & c_k  \\
\end{block}
\end{blockarray}}$, with the convention that $c_i = 0$ if this coefficient does not appear in $F_2$. Then, 
the corresponding elimination matrix is
\[
\mathbb{H}_{(2,1)}=
\begin{pmatrix} \Transpose{
a_0 & a_1 & a_2 & a_3 & a_4 \\
b_0 & b_1 & b_2 & b_3 & b_4 \\
c_0 & c_1 & 0 & c_3 & 0 \\
0 & c_0 & c_1 & 0 & c_3 \\
\left[013\right] & \left[023\right] + [014] & 0 & [024] & 0 
}
\end{pmatrix}.
\]
\lb{This example illustrates that we obtain the same type of matrices as in \cite{cattani1997residuesresultants} but under different assumptions: we are assuming that the polytopes $\Delta_0,\dots,\Delta_n$ are $n$-dimensional and $X_{\Sigma}$ has the $\sigma$-positive property, and in \cite{cattani1997residuesresultants} it is assumed that the $\alpha_i$'s are scaled copies of the same ample class.}

\begin{figure}[H]
    \centering
\begin{tikzpicture}

\draw[blue] (0,0) -- (1,0);
\draw[blue] (1,0) -- (0,1);
\draw[blue] (0,0) -- (0,1);

\filldraw[red] (0,0) circle (2pt) node[anchor=north] {\tiny (0,0)};
\filldraw[red] (1,0) circle (2pt) node[anchor=north] {\tiny (1,0)};
\filldraw[red] (0,1) circle (2pt) node[anchor=south] {\tiny (0,1)};

\end{tikzpicture}
    \caption{The polytope corresponding to the generic sparse homogeneous polynomial $F_2$ in Example \ref{example:h3}.}
    \label{fig:polytope2}
\end{figure}
\end{example}

As illustrated in Example \ref{example:h2}, the construction of elimination matrices $\mathbb{H}_\alpha$ requires the computation of the support of the local cohomology modules $H^i_{\mathfrak{b}}(R)$. This task can be delicate, although several results are known; \lb{for instance, see  \cite{altmann2018immaculate} for the cases where the fan $\Sigma$ splits or the rank of $\Pic(X_{\Sigma})$ is $2$ or $3$, or see also \cite{eisenbud2000cohomology,botbol2011implicit}}. In order to avoid such computations, the next result yields some sufficient conditions to get hybrid elimination matrices. 
\medskip

We recall that we are using the notation in Section \ref{sec:duality}. \cc{In particular, for $i = 0,\dots,n$}, we write $\alpha_i \in \Pic(X_{\Sigma})$ for the classes associated to the homogeneous polynomial system, $K_X$ for the anticanonical divisor \lb{and we set $\delta = \alpha_0 + \dots + \alpha_n - K_X$}. 

\begin{theorem}
\label{theo:deltapmnef} 
Assume that the toric variety $X_\Sigma$ is smooth and $\sigma$-positive for some maximal cone $\sigma \in \Sigma(n)$. Moreover, 
assume that \cc{the polytopes $\Delta_0,\dots,\Delta_n$ are $n$-dimensional. }If $\alpha \in \Pic(X_{\Sigma})$ satisfies either of the two following properties:
\begin{itemize}
    \item[i)] $\alpha = \delta + \nu$ with $\nu$ a nef class or,
    \item[ii)] $\alpha = \delta - \nu$, where $\nu$ \cc{is a nef class satisfying the hypotheses of Theorem \ref{theoremdecompositionrestated} and} \lb{for all $i = 0,\dots,n$, $\alpha_i - \nu$ is a nef class that corresponds to an $n$-dimensional polytope},
\end{itemize}
then $\mathbb{H}_\alpha$ is an elimination matrix. In addition, it is purely of Macaulay-type if and only if $\alpha$ satisfies \lb{i) but not ii)}. 
\end{theorem}
\begin{proof} First, \lb{recall that the notation $K_j(F)$ stands for the terms of the Koszul complex associated} to $F_0,\ldots,F_n$. We will also denote by $J$ subsets of $\{0,\dots,n\}$. For both cases, our strategy is to show that \cc{$\alpha \notin \Gamma_0 \cup \Gamma_1$ and $(I^{\sat}/I)_{\alpha} = \oplus_{\mu}A$ in order to apply Theorem \ref{prop:Helimmat}}. 

We begin with the case i) and pick $\alpha = \delta + \nu$ with $\nu$ a nef class. We have
\[H_{\mathfrak{b}}^i(K_j(F))_{\delta+\nu} \simeq H_{\mathfrak{b}}^i(\oplus_{|J| = j}C(-\sum_{l \in J}\alpha_l))_{\delta + \nu} \simeq \oplus_{|J| = j} H_{\mathfrak{b}}^i(C)_{\delta + \nu -\sum_{l \in J}\alpha_l} \quad i \geq 0, \: j = 0,\dots,n+1. \] \cc{Recall that} the local cohomology functors \lb{commute with direct sums}. Using \eqref{localsheaf} and \eqref{serreduality} for $i \geq 2$, we get:
\[ H_{\mathfrak{b}}^i(C)_{\delta + \nu -\sum_{l \in J}\alpha_l} \simeq H^{i-1}(X_{\Sigma},\sum_{l \notin J}\alpha_l - K_X + \nu) \simeq H^{n -i+1}(X_{\Sigma},-\sum_{l \notin J}\alpha_l - \nu).\]
\cc{If $J \neq \{0,\dots,n\}$ then the class of $\sum_{l \notin J}\alpha_l + \nu$ is nef and its associated polytope is $n$-dimensional. In this case, we can apply Theorem \ref{batyrevborisov} to deduce that 
$$H^{n -i+1}(X_{\Sigma},-\sum_{l \notin J}\alpha_l - \nu)=0 \textrm{ for } i\geq 2.$$
As  $H^i_{\mathfrak{b}}(C) = 0$ for $i = 0,1$ (see \eqref{remarkglobalsections}), it follows that
$$H_{\mathfrak{b}}^i(K_j(F))_{\delta+\nu} = 0 \textrm{ for all } j \neq n + 1$$
and hence, by definition of $\Gamma_0$ and $\Gamma_1$ (see \eqref{supportscohomology}), that $\delta + \nu \notin \Gamma_0 \cup \Gamma_1$. As a consequence, Theorem \ref{dualitytheorem} shows that:
$$ (I^{\sat}/I)_{\delta + \nu} = \Hom_A((C/I)_{-\nu},A).$$
As $\nu$ is nef (and also effective), Remark \ref{remarknegative} implies that $(C/I)_{-\nu}=0$ for all $\nu\neq 0$. On the other hand, Corollary \ref{corollarydecomposition} implies that and $(C/I)_{0}=A$. Therefore, 
$(I^{\sat}/I)_{\delta + \nu}=0$ for all $\nu$, except $\nu= 0$ where  we have $(I^{\sat}/I)_{\delta }\simeq A$. From here, Theorem \ref{prop:Helimmat} implies i).}

We proceed similarly to prove ii) and pick $\alpha = \delta - \nu$.  We have: \[ H_{\mathfrak{b}}^i(K_j(F))_{\delta-\nu} \simeq H_{\mathfrak{b}}^i(\oplus_{|J| = j}C(-\sum_{l \in J}\alpha_l))_{\delta - \nu} \simeq \oplus_{|J| = j} H_{\mathfrak{b}}^i(C)_{\delta - \nu -\sum_{l \in J}\alpha_l} \quad i \geq 0, \: j = 0,\dots,n+1 \]
and  using \eqref{localsheaf} and \eqref{serreduality}, for all $i\geq 2$ we get:
\[ H^{i}_{\mathfrak{b}}(C)_{\delta - \nu - \sum_{l \in J}\alpha_l} \simeq H^{i-1}(X_{\Sigma}, \sum_{l \notin J}\alpha_l - K_X - \nu) \simeq H^{n - i + 1}(X_{\Sigma}, \nu - \sum_{l \notin J}\alpha_l). \]
Our assumptions imply that the $\alpha_i - \nu$ are nef and their associated polytopes are $n$-dimensional for $i = 0,\dots,n$. Hence, if \cc{$J \neq \{0,\dots,n\}$, the classes $\sum_{l \notin J}\alpha_l - \nu$ are also nef and their associated polytopes are $n$-dimensional. Applying Theorem \ref{batyrevborisov}, we deduce that $$H_{\mathfrak{b}}^i(K_j(F))_{\delta-\nu} = 0 \textrm{ for all } j \neq n + 1,$$
\cc{hence $\delta - \nu \notin \Gamma_0 \cup \Gamma_1$. Applying Theorem \ref{dualitytheorem}, we deduce that}
 $(I^{\sat}/I)_{\delta - \nu} = \Hom_A((C/I)_{\nu},A).$
Finally, since $\nu$ satisfies the hypotheses of Theorem \ref{theoremdecompositionrestated}, we deduce from this theorem that $(C/I)_{\nu} = C_{\nu}$ and $(I^{\sat}/I)_{\delta - \nu}$ is a free $A$-module, which concludes the proof.}
\end{proof}
\begin{corollary}
Assume that the toric variety $X_\Sigma$ is smooth and $\sigma$-positive for some maximal cone $\sigma \in \Sigma(n)$. 
If \lb{the polytopes $\Delta_0,\ldots,\Delta_n$ are all} $n$-dimensional, then $\mathbb{H}_{\delta}$ is an elimination matrix.
\end{corollary}
\begin{proof} Apply Theorem \ref{theo:deltapmnef}, i) with $\nu=0$.
\end{proof}


\begin{example}
Taking again Example \ref{example:h2}, we \lb{observe} that \lb{several} elimination matrices are obtained from Theorem \ref{theo:deltapmnef}. \cc{Indeed, the matrix $\mathbb{H}_{(4,2)}$ is of Macaulay-type and corresponds to case i) in this theorem.
The matrix $\mathbb{H}_{(2,1)}$ corresponds to case ii) while the matrix $\mathbb{H}_{(3,1)}$ corresponds to both cases i) and ii) ($\nu = 0$). However, the matrix $\mathbb{H}_{(3,2)}$ does not belong to either of the two cases. Using the explicit computation of $\Gamma_0$ and $\Gamma_1$ that we showed in Figure \ref{figure1}, we can derive that $(I^{\sat}/I)_{(3,2)} = \Hom((C/I)_{-(0,1)},A)$ where $(C/I)_{-(0,1)} = 0$ and thus, $\mathbb{H}_{(3,2)}$ is also an elimination matrix.}
\end{example}

\subsection{Overdetermined sparse polynomial systems}
\label{sec:overdetermined}

In this section we extend the construction of hybrid elimination matrices to the case of homogeneous polynomial systems that are defined by $\ccc{t} + 1$ equations with $\ccc{t} \geq n$. \lb{Such systems often appear in practical applications and are referred to as overdetermined polynomial systems}. 


\begin{notation} We assume that the projective toric variety $X_{\Sigma}$ is smooth and $\sigma$-positive for some maximal cone $\sigma$. 
 In what follows, $F_0,\dots,F_{t}$ are generic homogeneous sparse polynomials corresponding to nef classes $\alpha_0,\dots,\alpha_{t}$, $I$ denotes the ideal they generate and $B = C/I$ the corresponding quotient ring. For each subset $T \subset \{0,\dots,\ccc{t}\}$ of cardinality $n+1$, we set $I_T = (F_i \, : \, i \in T)$, $B_T = C/I_T$ and $\delta_T = \sum_{i \in T}\alpha_i - K_X$. We denote by  \lb{$\Sylv_{T,\mu}$} the Sylvester forms that can be formed from $\{F_i\}_{i\in T}$; see Section \ref{sec:sylvforms}. We also denote by $K_{\bullet}(F)$ the Koszul complex of $F_0,\ldots,F_{t}$ and by $K_{T,\bullet}(F)$ the Koszul complex \lb{of $\{F_i\}_{i\in T}$}.
\end{notation}

The following result is a generalization of \cite[Chapter 3, Proposition 3.23]{busebook} which deals with the particular case $X_\Sigma=\mathbb{P}^n$.

\begin{theorem}
\label{theoremoverdetermined} 
Using the previous notation, suppose that there exists a subset $S \subset \{0,\dots,\ccc{t}\}$ of cardinality $n+1$ and a nef class $\nu \in \Pic(X_{\Sigma})$ satisfying the hypotheses of Theorem \ref{theoremdecompositionrestated} such that
\[ 
\forall i \in S \quad j \notin S \quad \alpha_i - \alpha_j \text{ is nef and }  \forall i \in S \quad \alpha_i - \nu \cc{\text{ is nef and corresponds to an } n\text{-dimensional polytope}}.\]
Then, the set of Sylvester forms
\[\{\lb{\sylv_{T,\mu}} \, : \,  T \subset \{0,\dots,\ccc{t}\} \textrm { such that } |T| = n+1 \textrm{ and } x^\mu \in C_{\delta_T - \delta_S + \nu}\}\]
yields a generating set of the $A$-module $(I^{\sat}/I)_{\delta_S - \nu}$.  
\end{theorem}

\begin{proof} First, we use Serre duality and Theorem \ref{batyrevborisov} in order to compute the local cohomology modules $H_{\mathfrak{b}}^i(K_j(F))_{\delta_S - \nu}$, for $i = 0,\dots,n+1$ and $j = 0,\dots,n$, similarly to what we did in Theorem \ref{theo:deltapmnef}. Namely, for $i \geq 2$ we get
$$H_{\mathfrak{b}}^i(K_j(F))_{\delta_S - \nu} \simeq \oplus_{|T| = j}H_{\mathfrak{b}}^i(C(-\sum_{k \in T}\alpha_k))_{\delta_S - \nu} \simeq \oplus_{|T| = j}H^{n + 1 - i}(X_{\Sigma}, \sum_{k \in T}\alpha_k - \sum_{k' \in S}\alpha_{k'} + \nu). $$
As we assumed that $j < n + 1$, we can show that the previous cohomology module is of the form $H^{n + 1 -i}(X_{\Sigma},-\alpha)$ for $\alpha$ a sum of nef divisors \cc{whose corresponding polytope is $n$-dimensional}. Namely, the elements in $S \cap T$ cancel each other, and the rest of elements $k' \in S$ can be either (i) paired up with $\alpha_k$ for $k \in T$ satisfying that $\alpha_k - \alpha_{k'}$ is nef, (ii) paired up with $\nu$ satisfying that $\alpha_{k'} - \nu$ is nef \cc{and the corresponding polytope is $n$-dimensional}, or (iii) they are nef themselves. Therefore, applying Theorem \ref{batyrevborisov} for $i \geq 2$ and Remark \ref{remarkglobalsections} for $i = 0,1$, we deduce:
\begin{equation}\label{lasttimeiwritethis}H_{\mathfrak{b}}^i(K_j(F))_{\delta_S - \nu} \simeq  0 \quad i = 0,\dots, n+1, \: j < n + 1. \end{equation}
As a consequence, from the comparison of the two spectral sequences that are considered in Theorem \ref{dualitytheorem}, we obtain the following transgression map, which is an isomorphism of graded modules: 
$$\tau: H_{n+1}(K_{\bullet}(F),H^{n+1}_{\mathfrak{b}}(C))_{\delta_S - \nu} \xrightarrow{\sim} H^0_{\mathfrak{b}}(B)_{\delta_S - \nu}.$$ 
For any $T \subset \{0,\dots,\ccc{t}\}$ of cardinalty $n + 1$, let $\tau_T$ be the corresponding transgression map for $K_{T,\bullet}(F)$ and $B_T$. For each of these Koszul complexes, we have a canonical morphism of complexes $K_{T,\bullet}(F) \xrightarrow[]{} K_{\bullet}(F)$ that induces:
\[ L_{\bullet}(F) = \bigoplus_{|T| = n + 1}K_{T,\bullet}(F) \xrightarrow[]{} K_{\bullet}(F).\]
It follows that there is a commutative diagram:
\begin{equation}\label{commdiagram}\begin{tikzcd}
\oplus_{|T| = n+1}H_{n+1}(K_{T,\bullet}(F), H_\mathfrak{b}^{n+1}(C))_{\delta_S - \nu} \arrow{r}{} \arrow[swap]{d}{\bigoplus_{T}\tau_T} & H_{n+1}(K_{\bullet}(F), H_\mathfrak{b}^{n+1}(C))_{\delta_S - \nu} \arrow{d}{\tau} \\
\oplus_{|T| = n+1}H^0_{\mathfrak{b}}(B_T)_{\delta_S - \nu} \arrow{r}{} & H_{\mathfrak{b}}^0(B)_{\delta_S - \nu}
\end{tikzcd}.\end{equation}
As the two vertical arrows are isomorphisms, in order to show that the bottom arrow is surjective, it is enough to show that the top arrow is surjective. For that purpose, we observe that $L_{n+1}(F) = K_{n+1}(F)$ by construction and also
\[ \oplus_{|T| = n+1}H_{n+1}(K_{T,\bullet}(F), H_\mathfrak{b}^{n+1}(C))_{\delta_S - \nu} = \ker(H_{\mathfrak{b}}^{n+1}(L_{n+1}(F)) \xrightarrow[]{} H_{\mathfrak{b}}^{n+1}(L_{n}(F)))_{\delta_S - \nu}. \]
However, by the same argument as in \eqref{lasttimeiwritethis}, $H_{\mathfrak{b}}^{n+1}(L_{n}(F))_{\delta_S - \nu} = 0$, so
\begin{equation}\label{equationl}\oplus_{|T| = n+1}H_{n+1}(K_{T,\bullet}(F), H_\mathfrak{b}^{n+1}(C))_{\delta_S - \nu} \simeq H_{\mathfrak{b}}^{n+1}(K_{n+1}(F))_{\delta_S - \nu}.\end{equation}
On the other hand,
\begin{multline*}
H_{n+1}(K_{\bullet}(F), H_\mathfrak{b}^{n+1}(C))_{\delta_S - \nu} \simeq \\ \ker(H_{\mathfrak{b}}^{n+1}(K_{n+1})_{\delta_S - \nu} \xrightarrow[]{} H_{\mathfrak{b}}^{n+1}(K_{n})_{\delta_S - \nu})/\im(H_{\mathfrak{b}}^{n+1}(K_{n+2})_{\delta_S - \nu} \xrightarrow[]{} H_{\mathfrak{b}}^{n+1}(K_{n+1})_{\delta_S -\nu}).  
\end{multline*}
As above, $H_{\mathfrak{b}}^{n+1}(K_{n})_{\delta_S-\nu} = 0$ and so
\begin{equation}
\label{equationkoszul}
H_{n+1}(K_{\bullet}(F), H_\mathfrak{b}^{n+1}(C))_{\delta_S -\nu} \simeq \\ H_{\mathfrak{b}}^{n+1}(K_{n+1})_{\delta_S - \nu}/\im(H_{\mathfrak{b}}^{n+1}(K_{n+2})_{\delta_S - \nu} \xrightarrow[]{} H_{\mathfrak{b}}^{n+1}(K_{n+1})_{\delta_S -\nu}).    
\end{equation}
By comparing \eqref{equationl} and \eqref{equationkoszul}, we see that the top map in the diagram \eqref{commdiagram} is surjective, as we wanted to prove. 
It follows that the basis of Sylvester forms of $\oplus_{|T| = n+1}H_{\mathfrak{b}}^0(B_T)_{\delta_S - \nu}$ is a set of generators of $H_{\mathfrak{b}}^0(B)_{\delta_S - \nu} = (I^{\sat}/I)_{\delta_S - \nu}$.
\end{proof}

We are now ready to extend the construction of hybrid elimination matrices to  overdetermined homogeneous polynomial systems. 


\begin{theorem}
\label{overdeterminedeliminationmatrix}We denote by $\mathbb{H}_{\alpha}$ the matrix of the following map:
\lb{
\begin{eqnarray}\label{hybrid-overdetermined}
    \left(\oplus_{i = 0}^nC(-\alpha_i)\right)_{\alpha} \bigoplus_{\substack{T \subset \{ 0,\dots,\ccc{t} \}, \, |T| = n+1,\\ x^\mu \in C_{\delta_T-\alpha }}} A & \rightarrow &  C_{\alpha}\\ \nonumber 
    (G_0,\ldots,G_n)\oplus (\ldots,\ccc{\ell}_{T,\mu},\ldots) &  \mapsto & \sum_{i = 0}^n G_iF_i + \sum_{\substack{T \subset \{ 0,\dots,\ccc{t} \}\\  |T| = n+1}}\sum_{x^\mu \in C_{\delta_T - \alpha}}\ccc{\ell}_{T,\mu}\Sylv_{T,\mu}	
\end{eqnarray}
where $\alpha=\delta_S-\nu$ and where $\ccc{\ell}_{\mu,T} \in A$ for all $\mu$ and $T$. Under the assumptions of Theorem \ref{theoremoverdetermined}, 
$\mathbb{H}_{\alpha}$ is an elimination matrix, where $\alpha=\delta_S-\nu$.
}
\end{theorem}

\begin{proof}
The proof goes along the same lines as the proof of Theorem \ref{prop:Helimmat} for the case $\ccc{t} = n$.
\end{proof}

\begin{example}
\label{example:h4}
Taking again the notation and the polynomials $F_0,F_1,F_2$ of  Example \ref{example:h2}, we add another polynomial of degree $\alpha_3 = (2,1)$ in $\mathcal{H}_1$ and write it in homogeneous coordinates as 
\[ F_3 = d_0z_1^2z_2 + d_1x_1z_1z_2 + d_2x_1^2z_2 + d_3x_2z_1 + d_4x_1x_2. \]
Following Theorem \ref{overdeterminedeliminationmatrix}, the matrix $\mathbb{H}_{\delta_S}$  for $\delta_S = (3,1)$ is
 \[
\centering
\setcounter{MaxMatrixCols}{13}
\begin{pmatrix} \Transpose{
a_0 & a_1 & a_2 & 0 & a_3 & a_4 & 0 \\
0 & a_0 & a_1 & a_2 & 0 & a_3 & a_4 \\
b_0 & b_1 & b_2 & 0 & b_3 & b_4 & 0 \\
 0 & b_0 & b_1 & b_2 & 0 & b_3 & b_4 \\
c_0 & c_1 & c_2 &  0 & c_3 & c_4 & 0 \\
0 & c_0 & c_1 & c_2 & 0 & c_3 & c_4 \\
d_0 & d_1 & d_2 &  0 & d_3 & d_4 & 0 \\
0 & d_0 & d_1 & d_2 & 0 & d_3 & d_4 \\
[130]_{abc} & [230]_{abc} & 0 & 0 & [430]_{abc} & 0 & 0 \\
[130]_{abd} & [230]_{abd} & 0 & 0 & [430]_{abd} & 0 & 0 \\
[130]_{acd} & [230]_{acd} & 0 & 0 & [430]_{acd} & 0 & 0 \\
[130]_{bcd} & [230]_{bcd} & 0 & 0 & [430]_{bcd} & 0 & 0 }
\end{pmatrix}
\]
where $[ijk]_{abc} = \scalemath{0.8}{
\begin{blockarray}{ccc}
\begin{block}{(ccc)}
 a_i & a_j & a_k \\
 b_i & b_j & b_k \\
 c_i & c_j & c_k  \\
\end{block}
\end{blockarray}}$, and $[ijk]_{abd}, [ijk]_{acd}, [ijk]_{bcd}$ defined accordingly. It is an elimination matrix for the overdetermined polynomial system defined by the polynomials $F_0,F_1,F_2$ and $F_3$.
\end{example}

We conclude this section with a comment on the computational impact of the hybrid elimination matrices obtained in Theorem \ref{overdeterminedeliminationmatrix}. Indeed, these matrices are intended for solving overdetermined 0-dimensional polynomial systems via eigenvalue and eigenvector computations, applicable over projective spaces, multi-projective spaces, or more broadly, smooth projective toric varieties that are $\sigma$-positive for a given maximal cone $\sigma$. In comparison with the more classical Macaulay-type matrices, hybrid elimination matrices are more compact. In particular, these matrices have a smaller number of rows, which is a key ingredient with respect to computational complexity. 
Indeed, this number of rows is controlled by the vanishing of the local cohomology modules at certain degrees, including the control of the saturation index of the homogeneous ideal $I(f)$ generated by general polynomials $f_0,\ldots,f_{t}$ of degrees $\alpha_0,\ldots,\alpha_{t}$. In the case of hybrid elimination matrices, the situation is similar with the difference that now one considers the homogeneous ideal generated by $f_0,\ldots,f_{t}$ and their toric Sylvester forms, whose saturation index is smaller than the one of $I(f)$. 

\ccc{Hybrid elimination matrices already appeared in the existing literature in various contexts and they are known to be more compact than Macaulay matrices. In order to be a little more concrete, we provide some specific examples in Table \ref{tab:numrows} to show explicitly the gain in terms of the number of rows.
We considered systems of 4 generic homogeneous  polynomials in four different settings of Newton polytopes and degrees, all corresponding to $3$-dimensional varieties.}

\begin{table}[h]

\begin{tabular}{||c || c | c|| c | c||} 
 \hline
 \multirow[c]{2}{*}[0in]{Type of system} & 
 \multicolumn{2}{c || }{degree $\alpha$}  &
 \multicolumn{2}{c ||}{number of rows} \\ \cline{2-5}
  & 
 Classical &  Hybrid  &
 Classical & Hybrid \\
 [0.5ex] 
 \hline\hline
 \text{Polynomials of deg.~$2$ in $\mathbb{P}^3$}& 5 & 3 &  56 & 20\\ 
 \hline
 \text{Polynomials of deg.~$10$ in $\mathbb{P}^3$} & 37 & 27 &  9880 & 4060 \\ 
 \hline
 \text{Polynomials of deg.~$(2,1)$ in $\mathbb{P}^2 \times \mathbb{P}^1$} & (6,3) & (4,2) &  112 & 45  \\
 \hline
 \text{Polynomials of deg. $\Delta \times [0,1]$ in $\mathcal{H}_1 \times \mathbb{P}^1$} & 3($\Delta \times [0,1]$) & 2($\Delta \times [0,1]$) & 88 & 36   \\
 \hline
\end{tabular}

\caption{The first column describes the type of system of 4 homogeneous polynomials which is considered. The second column provides the degree $\alpha$ for which the classical Macaulay-type matrices and the hybrid elimination matrices are constructed. The third column gives the corresponding number of rows of these two matrices. The Newton polytope $\Delta$ in the last row corresponds to the degrees of the polynomials considered in Example \ref{example:h2} .}
\label{tab:numrows}
\end{table}

\lb{We remark that} the number of columns of hybrid elimination matrices may increase fast when the number of equations is large compared to the dimension of the ground projective toric variety. Further work is needed to analyze if some toric Sylvester forms can be avoided or combined to gain in efficiency.  A more practical approach for  
future improvements would be to add Sylvester forms step by step (similarly to the ``degree-by-degree'' approach developed in \cite{yetbendertelen}) until the expected corank is achieved, or some other criterion needed to solve the polynomial systems is satisfied (see e.g.~\cite[Definition 2.1]{yetbendertelen}).

  \section{Sylvester forms and sparse resultants}\label{sec:resultants}

Resultants are central tools in elimination theory and there is a huge literature on various methods to compute them; see for instance  \cite{gkz1994,dandreadickenstein,weymannzelevinsky,dandrea2020cannyemiris,bender2021koszultype}. A classical result is that the sparse resultant can be computed as the determinant of certain graded components of the Koszul complex built from the considered polynomial system. In this section, we show that Sylvester forms can be incorporated in the  Koszul complex to obtain new expressions for the sparse resultant as the determinant of a complex. \cc{This construction extends results in  \cite[\S 2]{cattani1997residuesresultants} under the smoothness and $\sigma$-positive assumptions that we used in the previous sections}. 

In what follows, we assume that $X_{\Sigma}$ is a smooth projective toric variety which is $\sigma$-positive for some maximal cone $\sigma \in \Sigma(n)$. To begin with, we first recall briefly the definition of the sparse resultant. As in \eqref{nothomogenized}, given $F_0,\dots,F_n$ generic homogeneous sparse polynomials, we consider the dehomogenized polynomials $\Tilde{F}_i(x)$ by setting $z_1 = \dots = z_r = 1$. We recall that the supports of the $\Tilde{F}_i(x)$ are the $\mathcal{A}_i = \Delta_i \cap M$ for $i = 0,\dots,n$; see \eqref{nothomogenized}. The space of coefficients of the $\Tilde{F}_i$'s has a natural structure of multi-projective space $\prod_{i = 0}^n\mathbb{P}^{\mathcal{A}_i}$, as the \cc{zeros of} $\Tilde{F}_i = 0$ are not modified after multiplication by a nonzero scalar. Consider the incidence variety 
$$Z(\Tilde{F}) = \{x\times(\ldots,c_{i,m},\ldots) \in \mathbb{T}_N \times \prod_{i = 0}^n\mathbb{P}^{\mathcal{A}_i} \quad \Tilde{F}_0(x) = \cdots = \Tilde{F}_n(x) = 0\}$$ 
and let $\pi$ be
the canonical projection onto the second factor  
\begin{equation*}\pi: \mathbb{T}_N \times \prod_{i = 0}^n\mathbb{P}^{\mathcal{A}_i} \xrightarrow[]{} \prod_{i = 0}^n\mathbb{P}^{\mathcal{A}_i}.\end{equation*} 
 The \textit{sparse resultant}, denoted as $\Res_{\mathcal{A}}$, is a primitive polynomial defining the direct image $\pi_{*}(Z(\Tilde{F}))$. This polynomial is a power of an irreducible polynomial defining the closure of the image of $Z(\Tilde{F})$, i.e.~$\overline{\pi(Z(\Tilde{F})}$, if this is a hypersurface, and $1$ otherwise; see \cite[Section 3]{dandrea2020cannyemiris} where this irreducible polynomial is called the \textit{sparse eliminant}. Imposing that the $\Delta_i$'s are $n$-dimensional and the $\mathcal{A}_i$ span the lattice $M$ is sufficient for ensuring that the previous projection gives an hypersurface and that the sparse resultant is an irreducible polynomial; see \cite[Chapter 8]{gkz1994}. 


A classical method for computing the sparse resultant is to consider the determinant of the Koszul complex $K_\bullet(F)$ of the sequence of homogeneous  polynomials $F_0,\ldots,F_n$. \cc{An homogeneous version of the sparse resultant (see \cite[Chapter 3, Section 3]{gkz1994}) defined in terms of very ample sheaves (or equivalently ample in the smooth case \cite[Theorem 6.1.15]{coxlittleschneck}). Thus, when we need to use the properties of this homogeneous resultant, we will assume that the polytopes $\Delta_i$ \textit{correspond to ample divisors}; see \cite[Chapter 8, Proposition 1.5]{gkz1994} for a proof of the fact that $\Res_{\mathcal{A}}$ and this homogeneous resultant coincide.} Under these assumptions, we can compute $\Res_{\mathcal{A}}$ as the determinant of some graded pieces of the complex
\begin{equation}\label{koszulcomplex}
K_{\bullet}(F): K_{n+1} = C(-\sum \alpha_i) \xrightarrow[]{\partial_{n+1}} \dots \xrightarrow[]{\partial_3} K_2 = \oplus_{k,k'}C(-\alpha_k - \alpha_{k'}) \xrightarrow[]{\partial_2} K_1 =\oplus_{k}C(-\alpha_k) \xrightarrow{\partial_1} C.
\end{equation}
For all $\alpha \notin \Gamma_{\Res} \subset \Pic(X_{\Sigma})$, Remark \ref{rm:acyclicity} implies that the strand $K_\bullet(F)_\alpha$ is an acyclic complex of free $A$-modules and $H_0(K_\bullet(F)_\alpha)=B_\alpha$. Moreover, if we also consider $\alpha$ such that $(I^{\sat}/I)_{\alpha} = 0$, then $\det(K_\bullet(F)_\alpha)$ equals the sparse resultant $\Res_{\mathcal{A}}$ up to multiplication by a nonzero scalar; see \cite[Chapter 3, Theorem 4.2]{gkz1994} for proofs. 

In order to incorporate Sylvester forms in the above construction we proceed as follows. 
\lb{As in Definition \ref{def:H_alpha}, 
let $\alpha$ be such that $\left(I^{\sat}/I\right)_\alpha$ is a free $A$-module generated by Sylvester forms, i.e. $\left(I^{\sat}/I\right)_\alpha \simeq \oplus_{x^\mu \in C_{\delta - \alpha}} A$. We define the complex ${K^{\sat}_\bullet(F)}_\alpha$ as the graded strand $K_\bullet(F)_\alpha$ of the Koszul complex, where the map on  the right, namely $(K_1)_\alpha \rightarrow C_\alpha$ (see \eqref{koszulcomplex}), is replaced by the defining map \eqref{hybrid} of the hybrid elimination matrices. More precisely, ${K^{\sat}_\bullet(F)}_\alpha$ the following graded complex of free $A$-modules
\[ C(-\sum \alpha_i)_{\alpha} \xrightarrow[]{(\partial_{n+1})_\alpha} \ldots \xrightarrow[]{(\partial_3)_\alpha} \oplus_{k,k'}C(-\alpha_k - \alpha_{k'})_{\alpha} \xrightarrow[]{(\partial_2)_\alpha \oplus 0}  \oplus_{k}C(-\alpha_k)_{\alpha} \oplus_{x^\mu \in C_{\delta - \alpha}} A  \xrightarrow[]{(\partial_1)_\alpha \oplus \tau_\alpha} C_{\alpha},\]
where the map $(\partial_1)_\alpha \oplus \tau_\alpha$ is the map \eqref{hybrid}, $\tau_\alpha$ denoting the map from $\oplus_{x^\mu \in C_{\delta - \alpha}} A$ to $C_\alpha$ corresponding to the Sylvester forms. By definition, we notice that $H_i(K_{\bullet}^{\sat}(F)_{\alpha})=H_i(K_{\bullet}(F)_{\alpha})$ for all $i\geq 2$. Moreover, we also see that $H_1(K_{\bullet}^{\sat}(F)_{\alpha})\simeq H_1(K_{\bullet}(F)_{\alpha})$, because $\tau_\alpha$ is injective by property of the Sylvester forms, and that $H_0(K_{\bullet}^{\sat}(F)_{\alpha})=(B^{\sat})_\alpha$.
}

\begin{theorem} \label{resultantformula} Assume that $X_\Sigma$ is a smooth projective toric variety which is $\sigma$-positive for a maximal cone $\sigma$ \cc{and that the classes $\alpha_0,\dots,\alpha_n$ are ample}. For every $\alpha \notin \Gamma_{\Res}$, $K_{\bullet}^{\sat}(F)_{\alpha}$ is an acyclic complex of free $A$-modules. Moreover, if $\alpha = \delta - \nu$ as in Theorem \ref{theo:deltapmnef} ii), then $\det(K^{\sat}_\bullet(F)_\alpha)$ is equal to $\Res_{\mathcal{A}}$ up to a nonzero multiplicative scalar in $\mathbf{k}$. 
\end{theorem}

\begin{proof} \lb{By construction}, the acyclicity of $K_{\bullet}^{\sat}(F)_{\alpha}$ follows from the acyclicity of the usual Koszul complex (see above). \lb{Moreover, since } $H_0(K_{\bullet}^{\sat}(F)_{\alpha})=(B^{\sat})_\alpha$, we deduce that $\det(K_\bullet^{\sat}(F)_\alpha)$ and $\Res_{\mathcal{A}}$ are two polynomials in $A$ that vanish under the same specializations in $\mathbf{k}$. 
In order to show that they are the same polynomial, we will proceed by comparing their degrees \cc{with respect to the coefficients of $F_i$ for $i = 0,\dots,n$. For the sake of simplicity, we proceed by computing the degree of these polynomials with respect to the coefficients of $F_0$, and denote this degree as $\deg_{F_0}$.} As proved in \cite[Appendix A]{gkz1994},  the determinant of a complex of vector spaces $V_{\bullet}: V_{n+1} \xrightarrow[]{} \dots \xrightarrow[]{} V_1 \xrightarrow[]{} V_0$ is given by the formula 
\begin{equation}
\label{determinantofacomplex}
\det(V_\bullet) = \bigotimes_i\bigwedge^{\cc{\dim(V_i)}}V_i^{(-1)^i}.
\end{equation} 
Regarding the degree computation for $K_{\bullet}(F)$, the terms of the Koszul complex are $\mathbf{k}$-vector spaces tensored with $A$, thus we can apply \eqref{determinantofacomplex} to this complex of $A$-modules. The degree of $\bigwedge^{\dim(K_j(F))}K_j(F)_{\alpha}$ with respect to the coefficients of $F_0$ is:
\begin{equation}
    \sum_{J \subset \{0,\dots,n\}, \, |J| = j}^{0 \in J} \HF(R,\alpha - \sum_{j \in J}\alpha_j).
\end{equation}
For $\alpha \gg 0$, we have $(I^{\sat}/I)_{\alpha} = 0$ and $\HF(R,\alpha) = \HP(R,\alpha)$. Therefore, the degree of the determinant of the complex $K_{\bullet}^{\sat}(F)_\alpha$ coincides with the degree of the resultant and we can compute:

\begin{equation}
\label{degreeresultant}
    \deg_{F_0}(\Res_{\mathcal{A}}) = \deg_{F_0}\det(K^{\sat}_{\bullet}(F)_{\alpha}) = \sum_{J \subset \{0,\dots,n\}}^{0 \in J}(-1)^{|J|}\HF(R,\alpha - \sum_{j \in J}\alpha_j) = \sum_{J \subset \{0,\dots,n\}}^{0 \in J}(-1)^{|J|}\HP(R,\alpha - \sum_{j \in J}\alpha_j).
\end{equation}
As the degree of the resultant with respect to the coefficients of $F_0$ is constant (and equal to the mixed volume of the polytopes $\Delta_1,\dots,\Delta_n$), the last term in \eqref{degreeresultant} is a constant polynomial in $\alpha$, so when we evaluate it at any $\alpha$, it will always be equal to $\deg_{F_0}(\Res_{\mathcal{A}})$. Therefore, for $\alpha = \delta - \nu$ as in the statement, we have $(I^{\sat}/I)_{\delta - \nu} = \Hom_A(C_{\nu},A) \neq 0$ and we can check that the difference of degrees between the previous alternate sum and the degree of the classical Koszul complex is compensated by adding $(I^{\sat}/I)_{\delta-\nu}$ \cc{ at the term $K_{1}$ of $K_{\bullet}(F)$ (and thus counted with sign $-1$ in the determinant of the complex)}:
\[\deg_{F_0}\det(K_{\bullet}(F)_{\delta - \nu}) - \deg_{F_0}(\Res_{\mathcal{A}}) = \sum^{0 \in J}_{J \subset \{0,\dots,n\}}(-1)^{|J|}\big(\HF(R,\delta - \nu - \sum_{j \in J}\alpha_j) - \HP(R,\delta - \nu - \sum_{j \in J}\alpha_j)\big).\]
Using Grothendieck-Serre formula \eqref{grserre}, we deduce that this coincides with the quantity
\[\sum^{0 \in J}_{J \subset \{0,\dots,n\}}(-1)^{|J|}\sum_{i = 0}^{n+1}(-1)^i\dim_{\mathbf{k}} H^i_{\mathfrak{b}}(R)_{\delta - \nu - \sum_{j \in J}\alpha_j}.\]
Under the hypotheses of Theorem \ref{theo:deltapmnef} ii), and using Theorem \ref{batyrevborisov}, we get  that all the summands in the above sum vanish except if $i = n+1$ and $J = \{0,\dots,n\}$. In this latter case, we have $H^{n + 1}_{\mathfrak{b}}(R)_{- K_X - \nu}$, which is counted with the sign $(-1)^{2(n+1)} = 1$ \cc{and has the same dimension as the rank of the free $A$-module $H^{n + 1}_{\mathfrak{b}}(C)_{- K_X - \nu}$}. Recalling the duality theorem, which holds under the hypotheses of Theorem \ref{theo:deltapmnef} ii), we have:
$$H^{n + 1}_{\mathfrak{b}}(C)_{-K_X - \nu} \simeq (I^{\sat}/I)_{\delta - \nu} \simeq \Hom_A(C_{\nu},A),$$
which concludes the proof as the degree of each of the Sylvester forms $\Sylv_{\mu}$ for $x^{\mu} \in R_{\nu}$ with respect to the coefficients of $F_0$ is $1$.
\end{proof}

From the above result, we can also identify cases where the matrices $\mathbb{H}_{\alpha}$ are square matrices, and therefore their determinant (in the usual sense of the determinant of a matrix) is equal to the sparse resultant, up to a nonzero multiplicative constant. For this purpose, we consider \begin{equation}
\label{gammak2}
    \Gamma = \Supp (\oplus_{k,k'}C(-\alpha_k-\alpha_{k'}))
\end{equation} to be the support of the term $K_2(F)$ in the Koszul complex; see Figure \ref{figure1} for an example.

\begin{corollary}\label{cor:residue}
 Let $X_\Sigma$ be a smooth projective toric variety which is $\sigma$-positive for a maximal cone $\sigma$. Assume that $\Delta_0,\ldots,\Delta_n$ correspond to ample divisors. Then, for any $\alpha \notin \Gamma \cup \Gamma_{\Res} \cup \Gamma_0 \cup \Gamma_1$ we have $\det(\mathbb{H}_{\alpha}) = \Res_{\mathcal{A}}$, up multiplication by a nonzero scalar.
\end{corollary}
\begin{proof}
If $\alpha \notin \Gamma_{\Res} \cup \Gamma_0 \cup \Gamma_1$, then $(I^{\sat}/I)_{\alpha}$ is free and $\Res_{\mathcal{A}} = \det(K_{\bullet}^{\sat}(F)_{\alpha})$ as in Theorem \ref{koszulcomplex}. If $\alpha \notin \Gamma$, then the complex $K^{\sat}_{\bullet}(F)_{\alpha}$ has only two terms and therefore $\det(K^{\sat}_{\bullet}(F)_{\alpha}) = \det(\mathbb{H}_{\alpha})$.
\end{proof}

\begin{remark}
\label{remarkminors}
Computing the determinant of a complex can be done using some techniques such as Cayley determinants (see \cite[Appendix A]{gkz1994}), but it is not very practical. However, Theorem \ref{resultantformula} yields new expressions of the sparse resultant as a ratio of two determinants if $\alpha \notin \Supp \oplus_{k,l,m} C(-\alpha_k-\alpha_l-\alpha_m)$; see \cite[Corollary 2.4]{cattani1997residuesresultants} for a combinatorial characterization of such case. 
\end{remark}

We close this section with a comment and an example related to the well-known Canny-Emiris formula. For Macaulay-type \lb{matrices} of the form $\mathbb{M}_{\alpha}$, the Canny-Emiris formula gives a  way to choose a nonzero minor {of maximal size}; see \cite{cannyemiris93} for the formula and \cite{dandrea2020cannyemiris} for a proof  that this minor is nonzero.  It remains an open problem to see whether the conditions \lb{in} the proof of the Canny-Emiris formula \cite{dandrea2020cannyemiris} coincide with the Cayley determinant for such a choice of a minor. In the case of hybrid elimination matrices $\mathbb{H}_{\alpha}$,  a similar formula has been explored in \cite{emirisdandreabivariate} for $n = 2$ and $\alpha = \delta$.

\begin{example}
\label{exampleresultant}
Let's consider the four matrices provided in Example \ref{example:h2}, which correspond to the cases $\alpha \in \{ (4,2), (3,2), (3,1), (2,1)\}$. The last three are square matrices while the first one is not. We have drawn the region $\Gamma$ in brown in Figure \ref{figure1}, in order to indicate the elements that provide a square matrix, as well as $\Gamma_{\Res}$, in green, for the acyclicity of the complex. For the Macaulay-type matrices, we can combinatorially describe a maximal minor of $\mathbb{M}_{(4,2)}$ using the Canny-Emiris formula; see \cite{cannyemiris93,dandrea2020cannyemiris}.
The matrix $\mathbb{M}_{(3,2)}$ is square, 
\[ \mathbb{M}_{(3,2)} = \setcounter{MaxMatrixCols}{9}
 \begin{pmatrix} \Transpose{
 a_0 & a_1 & a_2 & 0 & 0 & a_3 & a_4 & 0 & 0 \\
 0 & a_0 & a_1 & a_2 &  0 & 0 & a_3 & a_4 & 0 \\
 0 & 0 & a_0 & a_1 & a_2 &  0 & 0 & a_3 & a_4  \\
 b_0 & b_1 & b_2 & 0 & 0 &  b_3 & b_4 & 0 & 0  \\
  0 & b_0 & b_1 & b_2 & 0 & 0 & b_3 & b_4 & 0  \\
 0 & 0 & b_0 & b_1 & b_2 & 0 & 0 & b_3 & b_4 \\
 c_0 & c_1 & c_2 &  0 & 0 & c_3 & c_4 & 0 & 0\\
 0 & c_0 & c_1 & c_2 &  0 & 0 & c_3 & c_4 & 0\\
 0 & 0 & c_0 & c_1 & c_2 &  0 & 0 & c_3 & c_4}
\end{pmatrix},
 \]
 and it  might be obtained using a greedy approach to the same formula (see \cite{cannypedersen, checaemiris}), \lb{but as far as we know, there was no known certificate of its existence as a resultant formula}. The hybrid matrices for $\alpha = (3,1),(2,1)$ are square. \lb{More generally, for non-square hybrid matrices, a procedure for choosing a minor is known when $n = 2$ and $\alpha = \delta$; see \cite{emirisdandreabivariate}.}

\end{example}

\section{Toric residue of the product of two forms}
\label{sec:residues}


Another topic for which Sylvester forms are of interest is the computation of toric residues. These objects were initially introduced by Cox as a way to relate the residue of a family of $n + 1$ forms to the integral of a certain \cc{differential} form in a toric variety $X_{\Sigma}$ (see \cite{coxtoricres}). Being given $F_0,\ldots,F_n$ generic homogeneous polynomials as in \eqref{eq:Fi}, and denoting by 
$K(A)$ the quotient field of the universal ring of coefficients $A$, Cox proved the existence of a residue map
\[\Residue_F: B_{\delta} \xrightarrow[]{} K(A)\]
(recall that $I=(F_0,\ldots,F_n)$ and $B=C/I$) which has the following property: for any specialization $\theta: A \xrightarrow[]{} \mathbf{k}$ (see Notation \ref{notSec5}) such that the specialized system $f_0 = \dots = f_n = 0$ has no solution in $X_{\Sigma}$, the residue map $\Residue_{f}: (R/I(f))_{\delta} \xrightarrow{} k$ is an isomorphism. Cox defined residue  maps through trace maps of Čech cohomology, but they can be characterized through the fact that, if there is no solution in $X_{\Sigma}$, $\rho(\sylv_0)$ is sent to $\pm 1 \in K$, so \cc{we can assume} $\Residue_F(\sylv_0) = \pm 1$. Many authors contributed formulas based on elimination matrices and resultants to compute residues  \cite{khetan2004combinatorial,dandreakhetanresidues,cattani1995residues,cattani1997residuesresultants} and also used them in other applications such as polynomial interpolation \cite{soprunovcodes} or mirror symmetry \cite{mirrorbatyrev}. In particular, in \cite{dandreakhetanresidues} an explicit formula for computing the toric residue of a form of degree $\delta$ as a quotient of two determinants ``\`a la Macaulay" is proved.

If a form $G$ of degree $\delta$ can be written as a product $G=PQ$, a natural question is to ask whether one can take advantage of this factorization in the computation of the residue of $G=PQ$ with respect to the polynomial system \lb{defined by} $F_0,\ldots, F_n$. In the case $X_\Sigma=\mathbb{P}^n$, Jouanolou proved  that this is possible by exploiting the duality between the degrees $\delta - \nu$ and $\nu$ of $P$ and $Q$,  respectively (see \cite[Proposition 3.10.27]{joanolouformesdinertie}\footnote{We notice that in \cite{joanolouformesdinertie} the residue is defined as a map to $A$, and not to $K(A)$, by multiplying with $\Res_{\mathcal{A}}$ in the image. \cc{In the case of ample divisors}, the product of the residue and the resultant lies in $A$; see \cite[Theorem 1.4]{cattani1997residuesresultants}.}). In what follows, we \cc{explore the generalization of} Jouanolou's formula to a general smooth projective toric variety $X_{\Sigma}$ which is $\sigma$-positive for a maximal cone $\sigma$, using toric Sylvester forms.

 Let $\mathbb{H}_{\delta - \nu}$ be an elimination matrix that satisfies the assumptions of Theorem \ref{theo:deltapmnef} ii), and let $\mathcal{H}_{\delta - \nu}$ be a nonzero maximal minor of $\mathbb{H}_{\delta - \nu}$ which contains the entire block built with Sylvester forms. Now, being given two generic forms $P \in C_{\nu}$ and $Q \in C_{\delta - \nu}$, we consider the matrix
\begin{equation}
\label{theta}
\Theta_{\delta- \nu} =
\begin{pmatrix}
\mathcal{H}_{\delta-\nu} & \mathbf{q} \\
\begin{matrix} 0 & \cc{(\mathbf{p})^T\mathcal{D}} \end{matrix} & 0
\end{pmatrix}
\end{equation}
where $\mathbf{p}$, respectively $\mathbf{q}$, stands for the vector of coefficients of $P$, respectively $Q$, \cc{and $\mathcal{D}$ is the matrix defined in \eqref{equationnotpairing}}. Recall that by the construction of the matrix $\mathbb{H}_{\delta - \nu}$, the matrix $\mathcal{H}_{\delta-\nu}$ is built as the join of a Macaulay-type block-matrix and another column-block matrix built from Sylvester forms. Thus, the row \cc{$(\mathbf{p})^T\mathcal{D}$} is aligned with the column-block built from Sylvester forms; see Example \ref{example:h4} for an illustration.

\medskip

 We first prove that the residue of the product of two monomials can be computed as a quotient. In what follows, we denote by \cc{
$\mathcal{H}_{\mu,\xi}$ the submatrix of $\mathcal{H}_{\delta - \nu}$ that is obtained by deleting the column corresponding to the monomial $x^{\mu} \in R_{\nu}$ and the row corresponding to the monomial $x^{\xi} \in C_{\delta-\nu}$. }
 \begin{lemma} Assume that $X_\Sigma$ is a smooth projective toric variety which is $\sigma$-positive for a maximal cone $\sigma$. 
Let $F_0,\dots,F_n$ be a system of homogeneous polynomials in $C$ as in \eqref{eq:Fi}, \cc{then for two monomials $x^{\mu} \in R_{\nu}$ and $x^{\xi} \in R_{\delta - \nu}$, 
\[\Residue_F(x^{\mu + \xi}) =  \frac{\sum_{x^{\mu'} \in C_{\mu} }(-1)^{\mu'}(-1)^{\xi}\mathcal{D}_{\mu,\mu'}\det(\mathcal{H}_{\mu',\xi})}{\det(\mathcal{H}_{\delta - \nu})},\]where $(-1)^{\mu'}$ (resp. $(-1)^{\xi}$) is set to $1$ if the relative position of the monomial $x^{\mu}$ (resp. $x^{\xi}$) in the columns (resp. rows) of $\mathcal{H}_{\delta - \nu}$ is even, otherwise it is set to $-1$.}
\end{lemma}
\begin{proof}
Let $H^{\xi}$ be the matrix obtained by multiplying the row of $\mathcal{H}_{\delta - \nu}$  corresponding to $x^\xi$ by the monomial $x^{\xi}$ itself. Then, by expanding the determinant along this row, one gets:
\[ x^{\mu}x^{\xi}\det(\mathcal{H}_{\delta - \nu}) = x^{\mu}\det(H^{\xi}) = x^{\mu}(\sum G_iF_i + \sum_{\mu' \in C_{\nu}}(-1)^{\mu}(-1)^{\xi}c_{\mu',\xi}\Sylv_{\mu'}). \]\cc{Then, using the matrix $\mathcal{D}$, we get that
$$x^{\mu}x^{\xi}\det(\mathcal{H}_{\delta - \nu}) = \sum x^{\mu}G_iF_i + \sum_{x^{\mu} \in C_{\mu} }(-1)^{\mu'}(-1)^{\xi}\mathcal{D}_{\mu,\mu'}c_{\mu',\xi}\Sylv_{0} \ \textrm{ modulo } I.$$ Taking residues, }we deduce that
\[\Residue_F(x^{\mu+\xi})\det(\mathcal{H}_{\delta-\nu}) =  \sum_{x^{\mu} \in C_{\mu} }(-1)^{\mu'}(-1)^{\xi}\mathcal{D}_{\mu,\mu'}c_{\mu,\xi}.\]
Finally, from the expansion of the determinant $\det(H^{\xi})$, we get that $c_{\mu',\xi} = \det(\mathcal{H}_{\mu',\xi})$. 
\end{proof}

We are now ready to prove the claimed formula for the residue of the product of two forms.
\begin{theorem}
\label{joanolouresidues}
Assume that $X_\Sigma$ is a smooth projective toric variety which is $\sigma$-positive for a maximal cone $\sigma$. 
Let $F_0,\dots,F_n$ be a system of homogeneous polynomials in $C$ as in \eqref{eq:Fi}, \lb{and suppose given} two forms $P \in C_{\nu}$ and $Q \in C_{\delta - \nu}$, \lb{then}
\[ \Residue_{F}(PQ) = \frac{\det(\Theta_{\delta - \nu})}{\det(\mathcal{H}_{\delta - \nu})}.\]
\end{theorem}

\begin{proof}
Write $P = \sum_{x^\mu \in C_\nu}p_{\mu}x^\mu$ and $Q = \sum_{x^\xi \in C_{\delta - \nu}}q_{\xi}x^{\xi}$. Then, by linearity of residues, we have:
\[ \Residue_F(PQ) = \sum_{x^\mu \in C_\nu, \, x^\xi \in C_{\delta - \nu}}p_{\mu}q_{\xi}\Residue_{F}(x^{\mu + \xi}) =  \frac{\sum_{\mu,\xi}\sum_{\mu'}\lb{(-1)^{\mu'}(-1)^{\xi}}p_{\mu}q_{\xi}\mathcal{D}_{\mu,\mu'}\det(\mathcal{H}_{\mu',\xi})}{\det(\mathcal{H}_{\delta - \nu})}.\]
The numerator is precisely the expansion of the determinant $\det(\Theta_{\delta - \nu})$ of the matrix defined in \eqref{theta}, firstly with respect to the last row and secondly with respect to the last column.
\end{proof}

\begin{example}
\label{example:residues}
In Example \ref{example:h2}, the elimination matrix \lb{$\mathbb{H}_{(2,1)}$} is square, therefore we take
\[ \mathcal{H}_{(2,1)} = \mathbb{H}_{(2,1)}=
\begin{pmatrix} \Transpose{
a_0 & a_1 & a_2 & a_3 & a_4 \\
b_0 & b_1 & b_2 & b_3 & b_4 \\
c_0 & c_1 & c_2 & c_3 & c_4 \\
\left[013\right] & \left[023\right] + [014] & [024] & 0 & 0 \\
\left[023\right] & [024] + [123] & [124] & 0 & 0
}
\end{pmatrix}.
\]
Let $P = p_0z_1 + p_1x_1$ and $Q = q_0z_1^2z_2 + q_1z_1z_2x_1 + q_2z_2x_1^2 + q_3z_1x_2 + q_4x_1x_2$ be homogeneous forms in $C_{(1,0)}$ and $C_{(2,1)}$, respectively \cc{ and let $\mathcal{D}$ be the matrix in Remark \eqref{equationnotpairing} which is of the form $\mathcal{D} = \scalemath{0.8}{
\begin{blockarray}{cc}
\begin{block}{(cc)}
1 & 0 \\
\mathcal{D}_{01} & 1 \\
\end{block}\end{blockarray}}$,}
then
\[ \Theta_{(2,1)} = 
\begin{pmatrix} \Transpose{
a_0 & a_1 & a_2 & a_3 & a_4 & 0 \\
b_0 & b_1 & b_2 & b_3 & b_4 & 0\\
c_0 & c_1 & c_2 & c_3 & c_4 & 0 \\
\left[013\right] & \left[023\right] + [014] & [024] & 0 & 0 & p_0 + \mathcal{D}_{01}p_1 \\
\left[023\right] & [024] + [123] & [124] & 0 & 0 & p_1 \\
q_0 & q_1 & q_2 & q_3 & q_4 & 0
}
\end{pmatrix}. \]
$\mathcal{D}_{01}$ can be computed as in \eqref{equationresidue} and it is nonzero as $z_1 \notin (x_1^2, x_2, z_1z_2)$. \lb{Applying} Theorem \ref{joanolouresidues}, we deduce that $\Residue_{F}(PQ) = \frac{\det(\Theta_{(2,1)})}{\det(\mathcal{H}_{(2,1)})}$.
For the sake of comparison, let us examine the formula we obtain by developing the product of $P$ and $Q$. In this case, we apply Theorem \ref{joanolouresidues} with $\delta = (3,1)$ and $\nu=0$, so we have to consider the matrix $\Theta_{(3,1)}$ which is of the form:
$$\Theta_{(3,1)} = \begin{pmatrix} \Transpose{
a_0 & a_1 & a_2 & 0 & a_3 & a_4 & 0 & 0 \\
0 & a_0 & a_1 & a_2 & 0 & a_3 & a_4 & 0\\
b_0 & b_1 & b_2 & 0 & b_3 & b_4 & 0 & 0\\
 0 & b_0 & b_1 & b_2 & 0 & b_3 & b_4 & 0\\
c_0 & c_1 & c_2 &  0 & c_3 & c_4 & 0 & 0\\
0 & c_0 & c_1 & c_2 & 0 & c_3 & c_4 & 0\\
[130] & [230] & 0 & 0 & [430] & 0 & 0 & 1 \\
p_0q_0 & p_0q_1 + p_1q_0 & p_0q_2 + p_1q_1 & p_1q_2 & p_0q_3 & p_0q_4 + p_1q_3 & p_1q_4 & 0
}
\end{pmatrix}$$
since the product $PQ$ is equal to
$$p_0q_0z_1^3z_2 + (p_0q_1 + p_1q_0)z_1^2z_2x_1 + (p_0q_2 + p_1q_1)z_1z_2x_1^2 + p_0q_3z_1^2x_2 
+ (p_0q_4 + p_1q_3)z_1x_1x_2 + p_1q_2z_2x_1^3 + p_1q_4x_1^2x_2.$$    
  The expansion of $\det(\Theta_{(3,1)})$ with respect to the last row leads to the same formula as in \cite[Corollary 3.4]{dandreakhetanresidues}.                                                
\end{example}

\section*{Acknowledgments}

\lb{The authors} are grateful to Matías R. Bender, Marc Chardin and Carlos D'Andrea for interesting discussions. \lb{They} also thank Ioannis Z. Emiris, Christos Konaxis and Elias Tsigaridas for their support as advisors of the second author. \lb{The authors are also grateful to the reviewer whose comments and suggestions helped to improve significantly this paper}. This project has received funding from the European Union’s Horizon 2020 research and innovation programme  under the Marie Skłodowska-Curie grant agreement No 860843. 

\printbibliography

\Addresses

\end{document}